\documentclass[11pt,a4paper]{amsart}
\usepackage{hyperref}
\usepackage{caption}
\usepackage[english]{babel}
\usepackage[utf8]{inputenc}
\usepackage{amsmath,amsthm,amssymb,amsfonts}
\usepackage{enumerate}
\usepackage{faktor}
\usepackage{textcomp}
\usepackage{graphicx}
\usepackage{extarrows}
\usepackage{epigraph}
\usepackage{graphicx}
\usepackage{subfigure}
\usepackage{sidecap}
\usepackage{indentfirst}
\usepackage{tikz}
\usepackage{pgfplots}
\usetikzlibrary{arrows,matrix,patterns,decorations.markings,positioning,shapes}
\usepackage{float}
\usepackage{color}
\usepackage{stmaryrd}
\usepackage{color}
\usepackage{pgfmath}
\usepackage{stmaryrd}
\usepackage{url} 

\usepackage{verbatim}

\newcommand{\R}{\mathbb{R}}

\newcommand{\Z}{\mathbb{Z}}

\renewcommand{\geq}{\geqslant}
\renewcommand{\leq}{\leqslant} 
\DeclareMathOperator{\spl}{sp}

\newtheorem{teo}{Theorem}[section]
\newtheorem*{teo*}{Theorem}
\newtheorem{lemma}[teo]{Lemma}
\newtheorem{prop}[teo]{Proposition}
\newtheorem*{prop*}{Proposition}
\newtheorem{defin}{Definition}
\newtheorem{remark}[teo]{Remark}
\newtheorem{cor}[teo]{Corollary}

\newtheorem{ex}[teo]{Example}
\newtheorem{question}{Question}

\pgfplotsset{compat=1.14}
\begin{document}
\title{Slice-torus concordance invariants and Whitehead doubles of links}

\author{Alberto Cavallo}
\address{Alfr\'ed R\'enyi Institute of Mathematics, 1053 Budapest, Hungary} 
\email{cavallo\_alberto@phd.ceu.edu} 
 
\author{Carlo Collari}
\address{Alfr\'ed R\'enyi Institute of Mathematics, 1053 Budapest, Hungary}
\email{carlo.collari.math@gmail.it}

\date{}

\maketitle

\begin{abstract}
In the present paper we extend the definition of slice-torus invariant to links. We prove a few properties of the newly-defined slice-torus link invariants:  the behaviour under crossing change, a slice genus bound, an obstruction to strong sliceness, and a combinatorial bound. Furthermore, we provide an application to the computation of the splitting number. Finally, we use the slice-torus link invariants, and the Whitehead doubling to define new strong concordance invariants for links, which are proven to be independent from the corresponding slice-torus link invariant.
\end{abstract}

\section{Introduction}

The study of knots up to concordance has quite some relevance in low-dimensional topology, and it has been extensively pursued. Two (smooth) knots in $\mathbb{S}^3$ are said to be \emph{concordant} if they bound a properly (smoothly) embedded annulus in $\mathbb{S}^3\times [0,1]$. 
A knot is called \emph{slice} if it is concordant to the unknot.
The set of knots up to concordance, endowed with the operation of connected sum, is an infinitely generated Abelian group $\mathcal{C}$, called the \emph{concordance group}, whose neutral element is the class of slice knots.

The advent of knot homologies (such as knot Floer homology and Khovanov-Rozansky homologies) led to the definition of a variety of new tools to study concordance, to obstruct sliceness and to compute the slice genus (i.e. the minimal genus of a surface properly embedded in $\mathbb{D}^4$ bounding the given knot). The earlier among these tools are the Ozsv\'ath-Szab\'o invariant $\tau$ (\cite{O-Sz}) and the Rasmussen invariant $s$ (\cite{Rasmussen}). 
These invariants (once suitably normalized) share lot of properties, and Livingston (\cite{Livingston}) decided to study them in a more general framework. Livingston defined a class of invariants, which were later named by Lewark (\cite{Lewark14}), as follows.

\begin{defin}[\cite{Lewark14, Livingston}]
A \emph{slice-torus invariant} is a knot concordance invariant $\nu$ such that:
\begin{itemize}
\item[$\triangleright$]  $\nu: \mathcal{C} \to \mathbb{R}$ is a group homomorphism;
\item[$\triangleright$]  $\vert \nu(K) \vert \leq g_{4}(K)$, where $g_{4}$ denotes the slice genus, for each knot $K$;
\item[$\triangleright$]  For each torus knot $T(p,q)$ we have that
\[ \nu(T(p,q)) =  \frac{(p-1)(q-1)}{2}\:.\]
\end{itemize}
\end{defin}

The family of the slice torus invariants includes, aside from $\tau$ and $s/2$, also (a suitable normalization of) the $\mathfrak{sl}_n$ ($n\geq 3$) analogues of the Rasmussen invariant. These invariants, denoted by $s_{n}$, were introduced independently by Lobb (\cite{Lobb09}) and Wu (\cite{Wu}). The $s_n$'s provide orthogonal information with respect to $\tau$ and $s$, and they were shown to be linearly independent from $s$ and $\tau$ by Lewark (\cite{Lewark14}).

The slice-torus invariants can be used to produce other concordance invariants. For example, using the fact that if two knots are concordant then also their Whitehead doubles are concordant, Livingston and Naik (\cite{LivingstonNaik}) defined\footnote{These function were originally defined only for integer-valued slice-torus invariants. Of course, the same definition works for all slice-torus invariants, and most of the properties proved in \cite{LivingstonNaik} still hold.} the functions
\[ F_{\nu}(K)(t) = \nu(W_{+}(K,t))\qquad \overline{F}_{\nu}(K)(t) = \nu(W_{-}(K,t)),\quad t\in \mathbb{Z}, \]
where $\nu$ is a slice-torus invariant and $W_{\pm}(K,t)$ denotes the positive (resp. negative) $t$-twisted Whitehead double of $K$. These functions are non-increasing, non-constant, take values respectively in $[0,1]$ and $[-1,0]$, and assume both the maximal and the minimal possible values. In particular, if the slice-torus invariant is integer-valued all the information contained in each function can be condensed into a single integer. These integers, denoted by $t_{\nu}$ and $\overline{t}_{\nu}$, are defined as the maximal value of $t$ such that $F_{\nu}(K,t) $ and $ \overline{F}_{\nu}(K,t)$, respectively, assume their maximum. It is not difficult to see that $\overline t_{\nu}(K) = -t_{\nu}(-K^*)-1$, where $-K^*$ is the mirror image of $K$ with the orientation reversed, so these invariants contain the same amount of information. At the time of writing it is still unknown whether the invariant $t_{\nu}$ can provide new information with respect to $\nu$. In fact there are some hints in the opposite direction; for instance, it is known that $t_{\tau} = 2\tau -1$ (\cite[Theorem 1.5]{Hedden}) and it has been conjectured that $t_{s/2} = 3s/2 - 1$ (\cite{Park}).

The aim of the present paper is to extend these definitions and constructions to the case of links, and to describe some applications and examples. Before stating the main results of this paper let us recall a few basic facts about link concordance. The first thing one should point out is that the definition of concordance is no longer unique. Two oriented links in $\mathbb{S}^3$ are said to be
\begin{itemize}
\item[$\triangleright$]  \emph{weakly concordant} if there exists a genus $0$ connected, compact, oriented surface, properly embedded in $\mathbb{S}^3 \times [0,1]$, bounding the two links;
\item[$\triangleright$]  \emph{strongly concordant} if there exists a disjoint union of annuli, properly embedded in $\mathbb{S}^3 \times [0,1]$, such that each of them bounds a component of each link.
\end{itemize}
In particular, strongly concordant links should have the same number of components. A link is said to be \emph{weakly} (resp. \emph{strongly}) \emph{slice} if is weakly (resp. strongly) concordant to an unlink.
Similarly, one can define a (\emph{weak}) \emph{slice genus} and a \emph{strong slice genus}. The former is just the minimal genus of any connected, compact, oriented surface properly embedded in $\mathbb{S}^{3}\times [0,1]$ bounding the link.  The latter has a similar definition but one has to consider only the surfaces such that each connected component bounds exactly one component of the link. The (weak) slice genus of a link shall be denoted by $g_{4}$.

Almost all the slice-torus invariants known to the authors can be extended to strong concordance invariants of links (see \cite{BeliakovaWehrli, Cavallo, Jeong}) and give rise to bounds on the slice genus. Thus far these invariants have been studied separately. Motivated by the common properties of these extended slice-torus invariants, in  Section \ref{sec:stli} we give the definition of slice-torus link invariants. For now, the reader should keep in mind that these are $\mathbb{R}$-valued strong concordance invariants, and that when restricted to knots these invariants give rise to $\R$-valued slice-torus invariants (Corollary \ref{corollary:restriction}). Moreover, the slice-torus link invariants include (once properly rescaled and translated) the extension to links of $\tau$ (\cite{Cavallo}), $s$ (\cite{BeliakovaWehrli}) and the $s_{n}$ invariants (\cite{Jeong}).
Some of the results in the present paper concerning the slice-torus link invariants were proved separately for $\tau$, $s$ and the $s_{n}$'s. We will mention whenever a result was known for one or more of the above-mentioned invariants, or if it is completely new.

The first result, which is known for $\tau$, $s$ and $s_{n}$ ($n\geq3$), consists of a bound on the slice genus and an obstruction to strong sliceness. 
\begin{prop}\label{corollary:bound_on_the_slice_genus}
If $\nu$ is a slice-torus link invariant and $L$ is an $\ell$-component link, then
\[-g_4(L)\leq  \nu(L) \leq g_{4}(L) + \ell -1 . \]
Furthermore, if $L$ is a strongly slice link, then $\nu(L)=0$.
\end{prop}
Another known property of $s$, $\tau$ and $s_{n}$ is the detection of the $3$-dimensional and, under the hypothesis of non-splitness, the slice genus of positive links. It turns out that these results holds for any slice-torus link invariant.
\begin{teo}\label{teorema:slice_genus_positive}
Let $L$ be an $\ell$-component positive link, and let $D$ be a positive diagram representing $L$. Then
\[ \nu(L) = g_{3}(L) + \ell - \ell_s = \frac{n(D) - O(D)  + \ell}{2},\]
for each slice-torus link invariant $\nu$, where $n(D)$ is the number of crossings of $D$, $O(D)$ denotes the number of Seifert circles and $\ell_s$ is the number of split components of $L$.
Futhermore, if $L$ is also non-split then we have that
\[ \nu(L)= g_{4}(L) + \ell - 1.\]
\end{teo}
Computing the value of slice-torus link invariants for non-positive links can be difficult. However we provide a combinatorial bound, whose proof appears in Section \ref{sec:stli}, which allows us to compute the slice-torus link invariants for certain classes of links, namely quasi-positive links (Theorem \ref{teo:quasi-pos}) and negative links (Proposition \ref{proposition:slice_genus_negative}).
This bound was known for the $s$-invariant (e.g. \cite{Cavallo4,Kawamura}), but unknown for $\tau$ and the $s_{n}$'s. Moreover, the value of the $s_{n}$'s for quasi-positive and negative links, and the value of $\tau$ for negative links were unknown (the value of $\tau$ for quasi-positive links was computed in \cite{Cavallo2}).
\begin{teo}\label{teorema:bound_combinatorio}
Let $L$ be an $\ell$-component link, let $\ell_s$ be the number of split components of $L$ and let $\nu$ be a slice-torus link invariant. For each non-splittable diagram \footnote{That is the number of connected components of $D$ is the number of split components of $L$. Equivalently, we cannot obtain another diagram for $L$ with more connected components than $D$.} $D$ representing $L$ the following inequality holds
\begin{equation}\label{equation:bound_combiantorio}
\frac{w(D) -O(D) + 2 s_{+}(D) + \ell -2 \ell_{s}}{2} \leq \nu (L).
\end{equation}
\end{teo}

The first truly novel application of the slice-torus link invariants is the computation of the splitting number. 
Before proceeding further, we wish to recall the reader the definition of the two main versions of the splitting number. Following \cite{Adams}, the \emph{splitting number} $\widetilde{\spl}$ is the minimal number of crossing changes (among all diagrams) necessary to transform an $\ell$-component link $L$ into the disjoint union of $\ell$ knots. A second version of the splitting number, which was studied for example in \cite{Friedletal,Cimasonietal}, has a similar definition but the only crossing changes allowed are those between different components. We denote this second version by $\spl$, and call it the \emph{strong splitting number}. Clearly, we have the inequality $\widetilde{\spl} \leq \spl$ but the equality does not hold in general. We prove that the slice-torus link invariants can be used to produce a lower bound for the splitting number $\widetilde{\spl}$.
\begin{teo} 
\label{theorem:splitting}
Suppose that $\nu$ is a slice-torus link invariant and $L$ is a link with components $K_1,...,K_{\ell}$. Then we have 
\begin{equation}
\label{equation:bound_on_splitting}
\left|\nu(L)-\sum_{i=1}^{\ell}\nu(K_i)\right|\leq\widetilde{\spl}(L)\:.
\end{equation}
\end{teo}
Furthermore, we provide an infinite family of examples where our bound on $\widetilde{\spl}$ is sharp, and $\widetilde{\spl} \ne \spl$ (Proposition \ref{proposition:splitting}). To the best of our knowledge, there is no other known method to compute the value of the splitting number for this family of links. A weaker version of Theorem \ref{theorem:splitting} featuring the strong splitting number was proved, for the $s_{n}$-invariants, in \cite{Jeong}.

\begin{remark}
Few months after the present paper was posted on the arxiv, the authors were informed that another bound on the splitting number $\widetilde{\spl}$ was previously discovered by A. Conway in his Ph.D.~thesis \cite[Proposition~4.4.5]{ConwayThesis}. Conway's bound is not published and does not appear in the arxiv, and uses completely different techniques from the ones used in the present paper. The main ingredients for Conway's bound are the multivariate signature and nullity. It can be checked that also Conway's bound can be used to compute the splitting number for the family $L_{t}$ in Proposition~\ref{proposition:splitting}. Nonetheless, we expect the two bounds to be independent. Since the comparison between the two bounds falls outside the scopes of the present paper, we leave the discussion of this topic to a forthcoming paper.
\end{remark}

The final part of our paper is dedicated to the definition of new strong concordance invariants via Whitehead doubling. The notion of Whitehead double for links is not unique. We will be interested in two different kinds of Whitehead doubles. The first kind is the fully clasped Whitehead double $W_\pm(L,\underline{m})$ which is basically obtained by doubling all the components, where $\underline{m}\in \mathbb{Z}^\ell$ encodes the number of twists in the double of each component.
The second type of Whitehead double we are interested in is the reduced Whitehead double $W^\prime_\pm(L,m;L_{1})$, which is obtained by doubling only the component $L_1$ inserting $m \in \mathbb{Z}$ full twists. Notice that in the case $L$ is a knot, the two constructions yield the same result: the $m$-twisted Whitehead double of $L$.

We use the fact that if two links are strongly concordant then their Whitehead doubles are also strongly concordant (Theorem \ref{teo:concordant}) to define four functions which are strong concordance invariants. More precisely, we consider the following functions
\[ F_{\nu}(L)(\underline{t})= \nu(W_{+}(L,\underline{t}))\quad\text{and}\quad F^\prime_{\nu}(L;L_{1})(t)= \nu(W^\prime_{+}(L,t;L_{1}))\]
and
\[ \overline{F}_{\nu}(L)(\underline{t})= \nu(W_{-}(L,\underline{t}))\quad\text{and}\quad \overline{F}^\prime_{\nu}(L;L_{1})(t)= \nu(W^\prime_{-}(L,t;L_{1})),\]
where $\nu$ is a slice-torus link invariant.
These functions generalize the functions $F_{\nu}(K)(t)$ and $\overline{F}_{\nu}(K)(t)$, and thus the invariant $t_{\nu}$ defined by Livingston and Naik (\cite{LivingstonNaik}).
\begin{teo}
Suppose that $\nu$ is a slice-torus link invariant, $L$ is a link and $L_1$ a component of $L$. The functions $F_{\nu}(L) $, $  F^{\prime}_{\nu}(L;L_{1}) $, $ \overline F_{\nu}(L)$, and $\overline F^{\prime}_{\nu}(L;L_{1})$ are non-increasing and bounded. Furthermore, $F_{\nu}(L)$ and $\overline F_{\nu}(L)$ are non-constant and assume the maximal possible value.
\end{teo}
As an application we show that these functions can be used to obstruct the existence of a strong concordance to a split link (Theorem \ref{teo:split}). We conclude the paper with some sample computations, proving the following result, which is still unknown in the case of knots.
\begin{teo}\label{teo:Fprimeandnu}
There exists a 2-component link $L$ and a slice-torus link invariant $\nu$ such that the function $F^\prime_{\nu}$ does not depend only on the linking matrix of $L$ and on  $\nu(L)$.
\end{teo}
\subsection*{Acknowledgements: }

A.C would like to thank Irena Matkovi\v{c} for her help during the writing of the paper.
A.C is supported by a Young Research Fellowship from the Alfr\'ed R\'enyi Institute of Mathematics.

C.C. wishes to thank Andr\'as Stipsicz and the Alfr\'ed R\'enyi Institute for the hospitality.
During the writing of the present paper C.C. was partially supported by an Indam scholarship for a research period outside Italy.
\section{Slice-torus link invariants}
\label{sec:stli}

In this section we introduce the slice-torus link invariants, and prove their first properties. We start by proving that slice-torus link invariants have a controlled behaviour with respect to the crossing change. This will be fundamental in the last part of the paper. Afterward, we prove the bound on the slice genus, which follows from a more general bound on the difference of the slice-torus link invariants of cobordant links. Finally, we compute the value of the slice-torus link invariants of the positive links, and we use it to produce the combinatorial bound. From the combinatorial bound will follow the computation of the value of slice-torus link invariants of the negative links.

\subsection{Definition and first properties}

Let us start with the definition of slice-torus link invariants.

\begin{defin}\label{defn:slice_torus_link_invariant}
A \emph{slice-torus link invariant} is an $\mathbb{R}$-valued strong concordance link invariant $\nu$ satisfying the following properties:
\begin{itemize}
\item[(A)] if $L_{1}$ and $L_2$ are related by an oriented band move, and $L_1$ has one component less than $L_{2}$ (cf. Figure \ref{figure:band_move}), then
\[ \nu (L_2) - 1 \leq \nu (L_1) \leq \nu (L_2); \]
\begin{figure}[h]
\centering
\begin{tikzpicture}[scale = .2]
\draw (0,4) circle (1 and 2);
\draw (0,-4) circle (1 and 2);

\draw[dashed] (8,2) arc (90:270:1 and 2);
\draw (8,-2) arc (-90:90:1 and 2);

\draw (0,-2) arc (-90:90:1 and 2);

\draw (0,6) .. controls  +(1,0) and +(-2,0) .. (8,2);
\draw (0,-6) .. controls  +(1,0) and +(-2,0) .. (8,-2);

\node at (4,-14) {$\vdots$};

\node at (12,-8) {$L_1$};
\node at (-4,-8) {$L_2$};

\draw (0,-10) circle (1 and 2);
\draw (0,-8) -- (8,-8);
\draw (0,-12) -- (8,-12);
\draw[dashed] (8,-8) arc (90:270:1 and 2);
\draw (8,-12) arc (-90:90:1 and 2);

\begin{scope}[shift = {+(0,-8)}]
\draw (0,-10) circle (1 and 2);
\draw (0,-8) -- (8,-8);
\draw (0,-12) -- (8,-12);
\draw[dashed] (8,-8) arc (90:270:1 and 2);
\draw (8,-12) arc (-90:90:1 and 2);
\end{scope}

\draw[dashed] (25,-4) circle (3);
\draw[dashed] (25,-16) circle (3);

\draw[->,>=stealth] (25,-8) -- (25, -12);

\draw[very thick, ->,>=stealth] (25+ 3*0.707,-4 + 3*0.707) .. controls  +(-1,-1) and +(-1,1) ..(25+ 3*0.707,-4-3*0.707) ;
\draw[very thick, <-,>=stealth] (25- 3*0.707,-4 + 3*0.707) .. controls  +(1,-1) and +(1,1) ..(25- 3*0.707,-4-3*0.707) ;

\draw[very thick, ->,>=stealth] (25+ 3*0.707,-16 + 3*0.707) .. controls  +(-1,-1) and +(1,-1) ..(25- 3*0.707,-16+3*0.707) ;
\draw[very thick, ->,>=stealth] (25- 3*0.707,-16 - 3*0.707) .. controls  +(1,1) and +(-1,1) ..(25+ 3*0.707,-16-3*0.707) ;
\end{tikzpicture}
\caption{A schematic description of a band move between the links $L_1$ and $L_2$ (left), and the local description of an oriented band move (right).}\label{figure:band_move}
\end{figure}
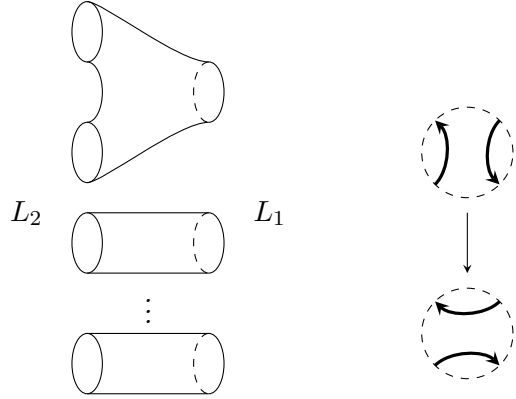
\item[(B)] $\nu$ is additive under disjoint union, that is $\nu (L_{1} \sqcup L_2 ) = \nu (L_1) + \nu (L_2)$;
\item[(C)] for each $\ell$-component link $L$ we have $$0\leq\nu (L)+\nu(-L^*)\leq \ell-1\:,$$ where $L^*$ denotes the mirror image of $L$, and $-L$ denotes $L$ with the orientation reversed;
\item[(D)] if $T_{p,q}$ is the positive $(p,q)$-torus knot, then
\[\nu(T_{p,q}) = \frac{(p-1)(q-1)}{2}\]
\end{itemize}
\end{defin}
Property (C) in the previous definition can be relaxed in the following sense: we may require the inequality in the property to hold only in the case of knots. Thus, we obtain that for each knot $K$ we have
\[\nu(-K^*)=-\nu(K).\] 
With this different definition in place, we can prove the following property:
\begin{itemize}
\item[($C^\prime$)] for every $\ell$-component link $L$, we have 
\[0\leq\nu(L)+\nu(-L^*)\leq\ell,\]
and $\nu(-K^*)=-\nu(K)$ for each knot $K$.
\end{itemize}
This is done by observing that we can obtain a strongly slice link if we perform $\ell$ band moves on the link $L\:\sqcup-L^*$. In turn, this can be seen by putting a diagram of $L$, and its mirror image with reverse orientation, in a symmetric position with respect to a line. Then, add an unknotted band between each pair of corresponding components. The result of this operation is a link bounding a ribbon surface which is the union of ribbon disks, and thus is a strongly slice link. Now ($C^\prime$) follows from Property (B). All the results in the paper, with the exception of Proposition \ref{prop:step} and some computations in Section \ref{section:example}, remain true replacing  Property (C) with Property ($C^\prime$) in the definition of slice-torus link invariant.

Our choice of Property (C) is motivated by the fact that, once suitably normalised, $s$ (\cite[Lemma 6.1]{BeliakovaWehrli}), the $s_{n}$'s (\cite[Theorem 3]{Jeong}) and $\tau$ satisfy it (see the examples below for the normalizations). For the latter invariant Property (C) follows from the additivity with respect to the connected sum (\cite[Subsection 3.3]{Cavallo}), and from the following lemma.
\begin{lemma}
 Suppose that $\nu$ is a strong concordance invariant that satisfies Properties (A), (B) and (D) and is additive under connected sums of links. Then Property (C) also holds and $\nu$ is a slice-torus link invariant.
\end{lemma}
\begin{proof}
 As we remarked before, we can apply $\ell$ bands move on $L\:\sqcup-L^*$, each one between a component of $L$ and its 
 corresponding mirror image, in the way that the result is strongly slice. We observe that the first of these moves changes 
 $L\:\sqcup-L^*$ into $L'=L\#-L^*$; then we have that $\nu(L')=\nu(L)+\nu(-L^*)$ by assumption.
 
 At this point, applying Property (A) $\ell-1$ times yields to \[0\leq\nu(L')=\nu(L)+\nu(-L^*)\leq\ell-1\:,\]
 where we used that $\nu(J)=0$ if $J$ is strongly slice. This last claim follows from the additivity of connected sums and Property (B).   
\end{proof}

\begin{ex}
\label{ex:first}
The quantity $\nu_s = \frac{s+\ell-1}{2}$ is a slice-torus link invariant, where $s$ is the extension of Rasmussen invariant (cf. \cite{Rasmussen}) to links introduced in \cite{BeliakovaWehrli} and $\ell$ is the number of components of the link.
\end{ex}
\begin{ex}
The Ozsv\'ath-Szab\'o $\tau$-invariant (cf. \cite{O-Sz}), which was extended to links in \cite{Cavallo}, is a slice-torus link invariant.
\end{ex}
\begin{ex}
More generally, we can consider the $\mathfrak{sl}_n$ version of the Rasmussen invariant, denoted with $s_n$, introduced by Lobb and Wu independently in \cite{Lobb,Wu}, which were extended to links in \cite{Jeong}. Then we have that $$\nu_{s_n} = \frac{-s_n(L)+(\ell-1)(n-1)}{2(n-1)}$$ is a slice-torus link invariant. In particular, if $n=2$ then $s_n(L)=-s(L)$ and we recover the expression in Example \ref{ex:first}.
\end{ex}

\begin{remark}
The unknot can be seen as $T_{1,p}$. In particular, it follows from Property (D) that for each slice-torus link invariant $\nu (\bigcirc) = 0$.
\end{remark}

The value of a slice-torus link invariant on the Hopf link, and the negative trefoil knot, is constant (i.e. does not depend on the slice-torus link invariant). Since these values shall be used in the follow up, we record them into the following lemma.

\begin{lemma}\label{lemma:values_of_nu_1}
Let $\nu$ be a slice-torus link invariant, and denote by $H_{\pm}$ the positive (resp. negative) Hopf link. Then, we have that $\nu(T_{2,3}^{*}) = -1$, $\nu(H_{+}) = 1$ and $\nu(H_{-})= 0$.
\end{lemma}
\begin{proof}
The first equality follows directly from Properties (C) and (D) in the definition of slice-torus invariant, and from the fact that $T_{2,3}^{*} = -T_{2,3}^{*}$. As concerns the other two equalities, notice that $H_{\pm}$ can be obtained from both the positive (resp. negative) trefoil knot and the unknot via a band move. Thus, it follows from Property (A) that
\[ 1 = \nu (T_{2,3}) \leq \nu (H_{+}) \leq \nu (\bigcirc) + 1 = 1.\]
A similar reasoning works for the negative Hopf link.
\end{proof}
Even though Property (C) implies that $\nu(-K^*)=-\nu(K)$ for each knot $K$, the previous lemma disproves the analogue of this result for multi-component links.

We now turn to another property of the slice torus link invariants: the additivity under connected sum of knots.
\begin{prop}\label{prop:additivity_for_knots}
Let $K_{1}$ and $K_{2}$ be two oriented knots. Then, for each slice-torus link invariant $\nu$ we have
\[\nu (K_{1} \# K_{2}) = \nu(K_1) + \nu(K_{2}),\]
where $\#$ denotes the connected sum.
\end{prop}
\begin{proof}
Since any connected sum can be obtained from a disjoint union via a band move, Properties (A) and (B) in Definition \ref{defn:slice_torus_link_invariant} tell us that
\[ \nu (K_{1} \# K_{2}) \leq \nu (K_{1} \sqcup K_{2}) =\nu(K_1) + \nu(K_{2}).\]
The same reasoning applied to $-K_{1}^*$ and $-K_{2}^*$, together with Property (C), implies that
\[ -\nu (K_{1} \# K_{2}) = \nu ( -(K_{1} \# K_{2})^* ) \leq \nu (-K_{1}^*) + \nu (- K_{2}^*) = -\nu(K_1) -\nu(K_{2}),\]
and the equality follows.
\end{proof}

\begin{remark}\label{remark:sub-additivity}
Denote by $L_{1} \#_{K_{1},K_{2}} L_{2}$ the connected sum of $L_{1}$ and $L_{2}$ along the components $K_{1}$ and $K_{2}$, respectively. Then, from Properties (A) and (B) it follows that
\[\nu(L_1) + \nu(L_{2}) - 1 \leq \nu (L_{1} \#_{K_{1},K_{2}} L_{2}) \leq \nu(L_1) + \nu(L_{2}),\]
for each slice-torus link invariant $\nu$.
\end{remark}

We conclude this subsection with the following proposition, concerning the behaviour of slice-torus link invariants under crossing changes. Recall that a \emph{cobordism} between two oriented links $L_{0}$ and $L_1$ is an oriented compact surface $\Sigma$, properly embedded in $\mathbb{S}^3\times[0,1]$, such that
\[\Sigma \cap \{ 0 \} = L_{0},\quad \text{and}\quad \Sigma \cap \{ 1\} = -L_{1},\]
where the orientation on the left-hand side of each equation is induced by $\Sigma$, and each connected component of $\Sigma$ has boundary on both $L_0$ and $L_1$. (This kind of cobordism is sometimes called \emph{good cobordism} in the literature, e.g. \cite{Rasmussen, BeliakovaWehrli}.)

\begin{prop}\label{proposition:crossing_change}
Let $D_{+}$ and $D_{-}$ be two link diagrams representing the links $L_{+}$ and $L_{-}$, respectively. Suppose that $D_{-}$ is obtained from $D_{+}$ by replacing a positive crossing with a negative one, then
\[ \nu( L_{-} )  \leq \nu(L_{+})  \leq \nu (L_{-}) + 1,\]
for each slice-torus link invariant $\nu$.
\end{prop}
\begin{proof}
The links $L_{+} \#_{K} T_{2,3}^*$ and $L_{-}$ are related by two band moves, where $K$ is a component of $L_{+}$ corresponding to a component of $D_{+}$ passing through the crossing changed. A movie describing the two band moves is shown in Figure \ref{figure:connected_sum_trefoil}.

The cobordism described in Figure \ref{figure:connected_sum_trefoil} can have either genus $0$ or genus $1$, depending on whether the strands involved in the first band move belong to different components or not. In both cases, combining the inequalities given by Property (A), we obtain
\[ \nu(L_{-})  - 1 \leq \nu (L_{+} \#_{K} T_{2,3}^*).\]
\begin{figure}[H]
\centering
\begin{tikzpicture}[scale =.35]

\draw[thick,->, >=stealth] (6,5) .. controls +(1,1) and +(3,0) .. (6,3) ;
\pgfsetlinewidth{20*\pgflinewidth}
\draw[white] (5,5.2) .. controls +(2,0) and +(-2,0) .. (8,2) ;
\pgfsetlinewidth{.05*\pgflinewidth}
\draw[thick,->, >=stealth] (5,5.2) .. controls +(2,0) and +(-2,0) .. (8,2) ;

\pgfsetlinewidth{20*\pgflinewidth}
\draw[white] (0,1) .. controls +(5,0) and +(-1,-1) .. (6,5) ;
\pgfsetlinewidth{.05*\pgflinewidth}
\draw[thick,->, >=stealth] (0,1) .. controls +(5,0) and +(-1,-1) .. (6,5) ;

\draw[thick] (6,5) .. controls +(.15,.15) and +(-.1,0) .. (6.42,5.2) ;
\pgfsetlinewidth{20*\pgflinewidth}
\draw[white] (0,4) .. controls +(3,0) and +(-4,0) .. (8,1) ;
\pgfsetlinewidth{.05*\pgflinewidth}
\draw[thick,->, >=stealth] (0,4) .. controls +(3.5,0) and +(-4.5,0) .. (8,1) ;

\pgfsetlinewidth{20*\pgflinewidth}
\draw[white] (6,3) .. controls +(-3,0) and +(-2,0) .. (5,5.2) ;
\pgfsetlinewidth{.05*\pgflinewidth}
\draw[thick,->, >=stealth] (6,3) .. controls +(-3,0) and +(-2,0) .. (5,5.2) ;

\draw[dashed] (3.3, 3.35) circle (1);

\draw[very thick,->, >=stealth] (8.5,3) -- (11.5,3);
\node at (10,3.75) {Band};
\node at (10,2.25) {move};
\begin{scope}[shift = {+(12,0)}]

\draw[thick,->, >=stealth] (6,5) .. controls +(1,1) and +(3,0) .. (6,3) ;
\pgfsetlinewidth{20*\pgflinewidth}
\draw[white] (5,5.2) .. controls +(2,0) and +(-2,0) .. (8,2) ;
\pgfsetlinewidth{.05*\pgflinewidth}
\draw[thick,->, >=stealth] (5,5.2) .. controls +(2,0) and +(-2,0) .. (8,2) ;
\draw[thick] (6,5) .. controls +(.15,.15) and +(-.1,0) .. (6.42,5.2) ;
\pgfsetlinewidth{20*\pgflinewidth}
\draw[white] (0,1) .. controls +(5,0) and +(-1,-1) .. (6,5) ;
\pgfsetlinewidth{.05*\pgflinewidth}
\draw[thick,->, >=stealth] (0,1) .. controls +(5,0) and +(-1,-1) .. (6,5) ;

\draw[thick,->, >=stealth] (0,3) .. controls +(3.5,0) and +(-1,0) .. (5,5.2) ;

\pgfsetlinewidth{20*\pgflinewidth}
\draw[white] (6,3) .. controls +(-6,0) and +(-3,0) .. (8,1) ;
\pgfsetlinewidth{.05*\pgflinewidth}
\draw[thick,->, >=stealth] (6,3) .. controls +(-6,0) and +(-3,0) .. (8,1) ;

\draw[dashed] (4, 2.5) circle (1.25);

\draw[very thick,->, >=stealth] (8.5,3) -- (11.5,3);
\node at (10,3.75) {$R_2$};
\node at (10,2.25) {move};
\end{scope}

\begin{scope}[shift = {+(24,0)}]

\draw[thick,->, >=stealth] (6,5) .. controls +(1,1) and +(3,0) .. (6,3) ;
\pgfsetlinewidth{20*\pgflinewidth}
\draw[white] (5,5.2) .. controls +(2,0) and +(-2,0) .. (8,2) ;
\pgfsetlinewidth{.05*\pgflinewidth}
\draw[thick,->, >=stealth] (5,5.2) .. controls +(2,0) and +(-2,0) .. (8,2) ;

\pgfsetlinewidth{20*\pgflinewidth}
\draw[white] (0,1) .. controls +(5,0) and +(-1,-1) .. (6,5) ;
\pgfsetlinewidth{.05*\pgflinewidth}
\draw[thick,->, >=stealth] (0,1) .. controls +(5,0) and +(-1,-1) .. (6,5) ;
\draw[thick,->, >=stealth] (0,3) .. controls +(3.5,0) and +(-1,0) .. (5,5.2) ;
\draw[thick] (6,5) .. controls +(.15,.15) and +(-.1,0) .. (6.42,5.2) ;
\draw[thick,->, >=stealth] (6,3) .. controls +(-6,0) and +(-3,0) .. (8,1) ;
\pgfsetlinewidth{20*\pgflinewidth}
\draw[white] (0,1) .. controls +(5,0) and +(-1,-1) .. (6,5) ;
\pgfsetlinewidth{.05*\pgflinewidth}
\draw[thick,->, >=stealth] (0,1) .. controls +(5,0) and +(-1,-1) .. (6,5) ;

\draw[dashed] (3, 3) circle (1);

\draw[very thick,->, >=stealth] (8.5,3) -- (11.5,3);
\node at (10,3.75) {Band};
\node at (10,2.25) {move};
\end{scope}
\begin{scope}[shift = {+(36,0)}]

\draw[thick,->, >=stealth] (6,5) .. controls +(1,1) and +(3,0) .. (6,3) ;
\pgfsetlinewidth{20*\pgflinewidth}
\draw[white] (5,5.2) .. controls +(2,0) and +(-2,0) .. (8,2) ;
\pgfsetlinewidth{.05*\pgflinewidth}
\draw[thick,->, >=stealth] (5,5.2) .. controls +(2,0) and +(-2,0) .. (8,2) ;

\draw[thick] (6,5) .. controls +(.15,.15) and +(-.1,0) .. (6.42,5.2) ;

\draw[thick,->, >=stealth] (0,4) .. controls +(3.5,0) and +(-4.5,0) .. (8,1) ;

\draw[thick,->, >=stealth] (6,3) .. controls +(-3,0) and +(-2,0) .. (5,5.2) ;
\pgfsetlinewidth{20*\pgflinewidth}
\draw[white] (0,1) .. controls +(5,0) and +(-1,-1) .. (6,5) ;
\pgfsetlinewidth{.05*\pgflinewidth}
\draw[thick,->, >=stealth] (0,1) .. controls +(5,0) and +(-1,-1) .. (6,5) ;

\end{scope}
\end{tikzpicture}
\caption{Two band moves relating $L_{+} \#_{K} T_{2,3}^*$ and $L_{-}$.}\label{figure:connected_sum_trefoil}
\end{figure}
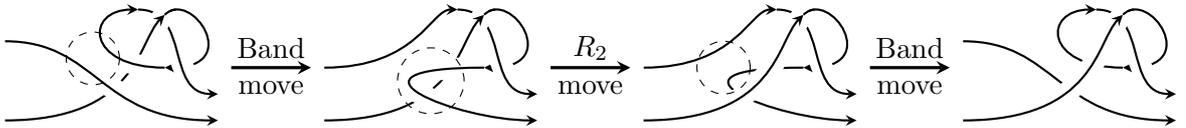
Putting together the above inequality, the inequalities in Remark \ref{remark:sub-additivity}, and the computations in Lemma \ref{lemma:values_of_nu_1}, we obtain
\[\nu(L_{-})  - 1 \leq \nu (L_{+} \#_{K} T_{2,3}^*) \leq  \nu (L_{+}) + \nu ( T_{2,3}^*) = \nu( L_{+}) -1,\]
which is the first inequality in the statement.
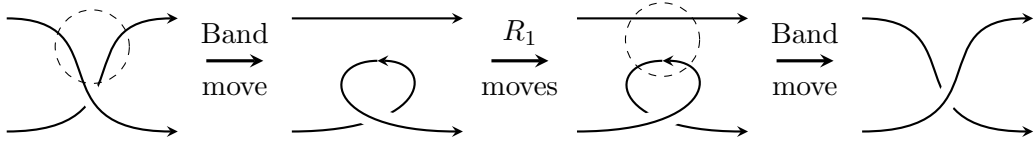
\begin{figure}[H]
\centering
\begin{tikzpicture}[scale = .75]

\draw[thick,->, >=stealth] (0,-.5) .. controls +(2.5,0) and +(-2,0) .. (3,1.5) ;
\pgfsetlinewidth{20*\pgflinewidth}
\draw[white](0,1.5) .. controls +(2,0) and +(-2.5,0) .. (3,-.5) ;
\pgfsetlinewidth{.05*\pgflinewidth}

\draw[thick,->, >=stealth] (0,1.5) .. controls +(2,0) and +(-2.5,0) .. (3,-.5) ;

\draw[dashed] (1.5,1) circle (.65);

\draw[very thick,->, >=stealth] (3.5,.75) -- (4.5,.75);
\node at (4,1.25) {Band};
\node at (4,.25) {move};

\begin{scope}[shift = {+(5,0)}]
\draw[thick,->, >=stealth] (0,-.5) .. controls +(2.5,0) and +(1,0) .. (1.5,.75) ;
\pgfsetlinewidth{20*\pgflinewidth}
\draw[white](1.5,.75) .. controls +(-1,0) and +(-2.5,0) .. (3,-.5) ;
\pgfsetlinewidth{.05*\pgflinewidth}

\draw[thick,->, >=stealth] (1.5,.75) .. controls +(-1,0) and +(-2.5,0) .. (3,-.5) ;

\draw[thick,->, >=stealth] (0,1.5) .. controls +(2,0) and +(-2.5,0) .. (3,1.5) ;
\draw[very thick,->, >=stealth] (3.5,.75) -- (4.5,.75);
\node at (4,1.25) {$R_{1}$};
\node at (4,.25) {moves};
\end{scope}

\begin{scope}[shift = {+(10,0)}]
\draw[thick,->, >=stealth] (1.5,.75) .. controls +(-1,0) and +(-2.5,0) .. (3,-.5) ;
\pgfsetlinewidth{20*\pgflinewidth}
\draw[white]  (0,-.5) .. controls +(2.5,0) and +(1,0) .. (1.5,.75) ;
\pgfsetlinewidth{.05*\pgflinewidth}

\draw[thick,->, >=stealth] (0,-.5) .. controls +(2.5,0) and +(1,0) .. (1.5,.75) ;

\draw[thick,->, >=stealth] (0,1.5) .. controls +(2,0) and +(-2.5,0) .. (3,1.5) ;

\draw[dashed] (1.5,1.125) circle (.65);
\draw[very thick,->, >=stealth] (3.5,.75) -- (4.5,.75);
\node at (4,1.25) {Band};
\node at (4,.25) {move};
\end{scope}

\begin{scope}[shift={+(15,0)}]
\draw[thick,->, >=stealth] (0,1.5) .. controls +(2,0) and +(-2.5,0) .. (3,-.5) ;
\pgfsetlinewidth{20*\pgflinewidth}
\draw[white](0,-.5) .. controls +(2.5,0) and +(-2,0) .. (3,1.5) ;
\pgfsetlinewidth{.05*\pgflinewidth}

\draw[thick,->, >=stealth]  (0,-.5) .. controls +(2.5,0) and +(-2,0) .. (3,1.5) ;

\end{scope}

\end{tikzpicture}
\caption{Two band moves relating $L_{+}$ and $L_{-}$.}\label{figure:crossing_change}
\end{figure}
To recover the second inequality, consider the cobordism in Figure \ref{figure:crossing_change}. There are two cases to consider, depending whether or not the first band move merges two components. In both cases, Property (A) tells us that
\[\nu (L_{+})  - 1 \leq \nu (L_{-}),\]
and the result follows. 
\end{proof}

\subsection{A bound on the slice-genus}

In analogy with the case of the slice-torus invariants, each slice-torus link invariant gives rise to a lower bound for the slice genus. This bound is a consequence of the following, more general, proposition.

\begin{prop}\label{proposition:slice_genus_bound}
Let $\nu$ be a slice-torus link invariant. Given two links $L_0$ and $L_1$ with $\ell_0$ and $\ell_{1}$ components, respectively, such that there exists a cobordism $\Sigma \subset \mathbb{S}^3 \times [0,1]$ from $L_0$ to $L_{1}$  with $k$ connected components, then 
\[ \nu (L_1) - g(\Sigma) -\ell_{1} +k  \leq \nu (L_0) \leq \nu(L_1) + g(\Sigma) + \ell_0 - k,\]
where $g(\Sigma)$ denotes the genus of $\Sigma$. In particular, when $\Sigma$ is connected we have
\[ \nu (L_1) - g(\Sigma) -\ell_{1} +1  \leq \nu (L_0) \leq \nu(L_1) + g(\Sigma) + \ell_0 - 1\:.\]
\end{prop}

\begin{proof}

By standard arguments, up to a boundary fixing ambient isotopy we may assume that the projection onto the second factor
\[pr_{2}:  \mathbb{S}^3 \times [0,1] \to [0,1], \]
when restricted to $\Sigma\setminus \partial \Sigma$ has only a finite number of (non-degenerate) critical values, let us denote these values by $0 < t_{0} <... < t_{ h}< 1$. Basic Morse theory tells us that we may assume the links $L^{\pm}_{i} = pr_{2}^{-1}(t_{i}\pm \epsilon) \cap \Sigma$, where $i \in \{ 0,..., h\}$ and $\epsilon>0$ is sufficiently small, to be obtained one from the other by either an oriented band move (1-handle attachment), the split union with an unknot (0-handle attachment), or the removal of an unknotted split component (2-handle attachment). Furthermore, we have that $pr_{2}^{-1}([t_{i}+\epsilon,t_{i+1} - \epsilon])$ is (topologically) a disjoint union of cylinders, and that $L_{i}^{+}$ and $L_{i+1}^{-}$ are isotopic.

Thanks to \cite[Theorem 3.1]{Kawauchietal}, up to isotopy, the order of the attachments can be chosen as follows
\begin{enumerate}
\item we start with $L_0 = \Sigma \cap  \mathbb{S}^{3} \times\{ 0 \}$;
\item we attach all the $0$-handles;
\item we perform a sequence of fusion 1-handles (i.e. 1-handles attachments lowering the number of components) merging all the newly attached $0$-handles;
\item we perform another sequence of $\ell - k$ fusion 1-handles until we end up with a $k$-component link diagram. Each fusion move merges two knots which belong to the same component of $\Sigma$;
\item we perform $g$ fission 1-handles (i.e. 1-handles attachments increasing the number of components), followed by $g$ fusion 1-handles (with $g = g(\Sigma)$);
\item we perform a sequence of fission 1-handles (and isotopies) ending up into the link obtained as a split union of $L_1$ and an unlink;
\item we attach all $2$-handles on the unlink;
\item we end up with $L_1 = \Sigma \cap \mathbb{S}^{3} \times \{ 1 \}$.
\end{enumerate}
A schematic representation of the cobordism $\Sigma$ when the surface is connected can be seen in Figure \ref{figure:cobordism}. As we highlighted in Figure \ref{figure:cobordism}, the cobordism obtained from the attachments described in point 0 to 3 (resp. in points 6 and 8) is a strong concordance between $L_0$ and a link $L_0^{\prime}$ (resp. a link $L_1^\prime$ and $L_1$).

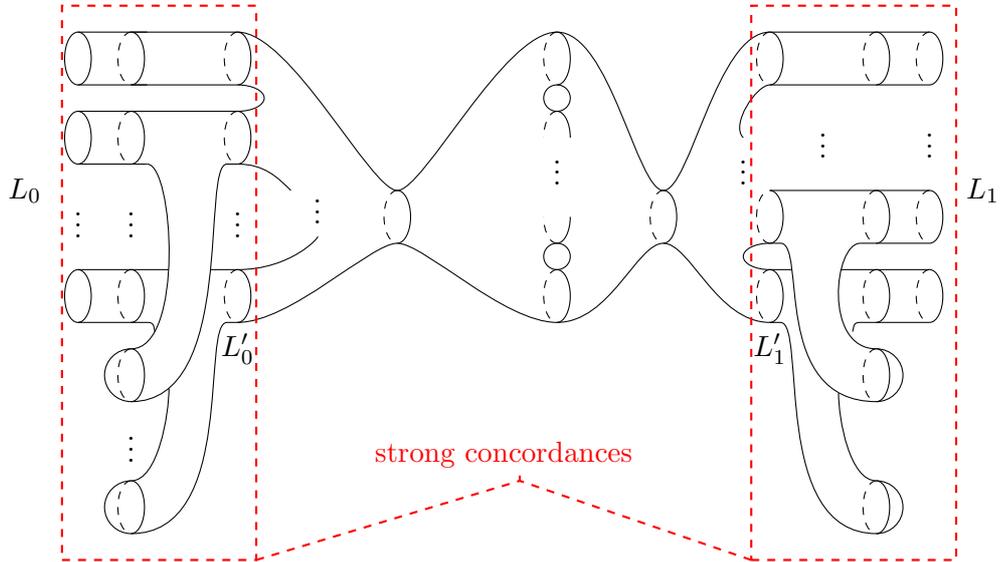
\begin{figure}[h]
\centering
\begin{tikzpicture}[scale = .35]
\draw (0, 0) circle (.5 and 1);
\draw (0,-3) circle (.5 and 1);
\draw (0,-9) circle (.5 and 1);

\node at (0,-6) {$\vdots$};

\draw (0,1) -- (6,1);
\draw (0,-1) -- (6,-1);
\draw (0,-2) -- (6,-2);
\draw (0,-4) -- (2.6,-4);
\draw (0,-8) -- (3.41,-8);
\draw (0,-10)-- (2.6,-10) .. controls +(.25,0) and +(-.05,.25) ..  (2.89, -10.35);
\draw (4.99,-8) -- (6,-8);

\draw (2, 1) arc (90:-90:.5 and 1);
\draw (2,-2) arc (90:-90:.5 and 1);
\draw (2,-8) arc (90:-90:.5 and 1);

\draw[dashed] (2, 1) arc (90:270:.5 and 1);
\draw[dashed] (2,-2) arc (90:270:.5 and 1);
\draw[dashed] (2,-8) arc (90:270:.5 and 1);

\node at (2,-6) {$\vdots$};

\draw (2, -11) arc (90:-90:.5 and 1);
\draw (2,-16) arc (90:-90:.5 and 1);

\draw[dashed] (2, -11) arc (90:270:.5 and 1);
\draw[dashed] (2,-16) arc (90:270:.5 and 1);

\draw (2,1) -- (2.6,1);
\draw (2,-1) -- (2.6,-1);

\node at (2,-14.5) {$\vdots$};

\draw (2, -11) arc (90:270:1);
\draw (2,-16) arc (90:270:1);

\node at (6,-6) {$\vdots$};

\draw (6, 1) arc (90:-90:.5 and 1);
\draw (6,-2) arc (90:-90:.5 and 1);
\draw (6,-8) arc (90:-90:.5 and 1);

\draw[dashed] (6, 1) arc (90:270:.5 and 1);
\draw[dashed] (6,-2) arc (90:270:.5 and 1);
\draw[dashed] (6,-8) arc (90:270:.5 and 1);

\draw (2, -11) .. controls +(2,0) and +(1,0) .. (2.6,-4);
\draw (2, -13) .. controls +(4,0) and +(-1,0) .. (5.6,-4)-- (6,-4);

\draw (2, -16) .. controls +(1,0) and +(0,-1) .. (3.42,-12.5);
\draw (2, -18) .. controls +(4,0) and +(-1,0) .. (5.6,-10)-- (6,-10);

\draw (12,-5) arc (90:-90:.5 and 1);

\draw[dashed] (12, -5) arc (90:270:.5 and 1);

\draw (6,    1) .. controls +(2,0) and +(-1,0) .. (12,-5);
\draw (6, -10) .. controls +(2,0) and +(-1,0) .. (12,-7);
\draw (6, -8) .. controls +(2,0) and +(.1,0) .. (9,-6.75);
\draw (6, -4) .. controls +(1,0) and +(-.5,.5) .. (8,-5);
\draw (6, -1) arc (90:-90:1 and .5);

\draw (18,    1) .. controls +(-2,0) and +(1,0) .. (12,-5);
\draw (18, -10) .. controls +(-2,0) and +(1,0) .. (12,-7);
\draw (18, -7) arc (90:270:.5);
\draw (18, -1) arc (90:270:.5);

\draw (18, 1) arc (90:-90:.5 and 1);
\draw (18,-2) arc (90:0:.5 and 1);
\draw (18,-8) arc (90:-90:.5 and 1);
\draw (18,-7) arc (-90:0:.5 and 1);

\draw[dashed] (18, 1) arc (90:270:.5 and 1);
\draw[dashed] (18,-2) arc (90:180:.5 and 1);
\draw[dashed] (18,-8) arc (90:270:.5 and 1);
\draw[dashed] (18,-7) arc (270:180:.5 and 1);

\draw (18,    1) .. controls +(2,0) and +(-1,0) .. (22,-5);
\draw (18, -10) .. controls +(2,0) and +(-1,0) .. (22,-7);
\draw (18, -7) arc (90:-90:.5);
\draw (18, -1) arc (90:-90:.5);

\draw (22,-5) arc (90:-90:.5 and 1);
\draw[dashed] (22, -5) arc (90:270:.5 and 1);

\draw (26,    1) .. controls +(-2,0) and +(1,0) .. (22,-5);
\draw (26, -10) .. controls +(-2,0) and +(1,0) .. (22,-7);
\draw (26, -7) arc (90:270:1 and .5);
\draw (26, -1) .. controls +(-.5,0) and +(-.5,.5) .. (25,-3);

\draw (26, 1) arc (90:-90:.5 and 1);
\draw (26,-8) arc (90:-90:.5 and 1);
\draw (26,-7) arc (-90:90:.5 and 1);

\draw[dashed] (26, 1) arc (90:270:.5 and 1);
\draw[dashed] (26,-8) arc (90:270:.5 and 1);
\draw[dashed] (26,-7) arc (270:90:.5 and 1);
%
%

%
\draw[dashed,red, thick] (-.6,2) rectangle (6.7,-19);
\draw[dashed,red, thick] (25.3,2) rectangle (33,-19);
\node at (16,-15) {\textcolor{red}{strong concordances}};
\draw[dashed,red, thick] (6.7,-19) -- (16.6,-16) -- (16.6, -15.75);
\draw[dashed,red, thick] (25.3,-19) -- (16.6,-16);

\node at (-2,-5) {$L_0$};
\node at (34,-5) {$L_1$};
\node at (26,-11) {$L_{1}^{\prime}$};
\node at (6,-11) {$L_0^{\prime}$};

\begin{scope}[xscale = -1, shift={+(-32,0)}]

\draw (0,1) -- (6,1);
\draw (0,-1) -- (6,-1);
\draw (0,-7) -- (2.6,-7);
\draw (0,-5) -- (6,-5);
\draw (0,-8) -- (3.35,-8);
\draw (0,-10)-- (2.6,-10) .. controls +(.25,0) and +(-.05,.25) ..  (2.89, -10.35);
\draw (5.18,-8) -- (6,-8);

\draw (2, -11) .. controls +(2,0) and +(1,0) .. (2.6,-7);
\draw (2, -13) .. controls +(4,0) and +(-1,0) .. (5.6,-7)-- (6,-7);

\draw (2, -16) .. controls +(1,0) and +(0,-1) .. (3.42,-12.68);
\draw (2, -18) .. controls +(4,0) and +(-1,0) .. (5.6,-10)-- (6,-10);
\end{scope}

\draw (30,1) arc (90:-90:.5 and 1);
\draw[dashed] (30,1) arc (90:270:.5 and 1);

\draw (30,-5) arc (90:-90:.5 and 1);
\draw[dashed] (30,-5) arc (90:270:.5 and 1);

\draw (30,-8) arc (90:-90:.5 and 1);
\draw[dashed] (30,-8) arc (90:270:.5 and 1);

\draw (30,-11) arc (90:-90:.5 and 1);
\draw[dashed] (30,-11) arc (90:270:.5 and 1);
\draw (30,-11) arc (90:-90:1);

\draw (30,-16) arc (90:-90:.5 and 1);
\draw[dashed] (30,-16) arc (90:270:.5 and 1);
\draw (30,-16) arc (90:-90:1);

\draw (32,1) arc (90:-90:.5 and 1);
\draw[dashed] (32,1) arc (90:270:.5 and 1);

\draw (32,-5) arc (90:-90:.5 and 1);
\draw[dashed] (32,-5) arc (90:270:.5 and 1);

\draw (32,-8) arc (90:-90:.5 and 1);
\draw[dashed] (32,-8) arc (90:270:.5 and 1);
\node at (28,-3) {$\vdots$};
\node at (32,-3) {$\vdots$};
\node at (9,-5.5) {$\vdots$};
\node at (18,-4) {$\vdots$};
\node at (25,-4) {$\vdots$};
\end{tikzpicture}
\caption{A schematic description of a connected cobordism $\Sigma$ after the re-ordering of the handles.}\label{figure:cobordism}
\end{figure}

By the strong-concordance invariance of $\nu$, it follows immediately that $\nu(L_0) = \nu(L_0^\prime)$ and $\nu(L_1) = \nu(L_1^\prime)$. Now consider the portion of $\Sigma$ between the link $L_0^\prime$ and $L_1^\prime$, say $\Sigma^\prime$, then it follows from Property (A) that
\[\nu(L_0^\prime)  \leq \nu(L_1^\prime) + g(\Sigma^\prime) +\ell_0 - k.  \]
Since $g(\Sigma^\prime) = g(\Sigma)$, the second inequality in the statement follows. The other inequality is obtained by reversing the roles of $L_0$ and $L_1$.
\end{proof}

Now, as an easy consequence of Proposition \ref{proposition:slice_genus_bound} we obtain the desired lower bound on the slice genus.
\begin{proof}[Proof of Proposition \ref{corollary:bound_on_the_slice_genus}]
Consider a minimal genus surface $\Sigma \subset \mathbb{D}^4$ bounding $L$ ($= \Sigma \cap \mathbb{S}^3$). Without altering the genus we may assume $\Sigma$ to be connected. By removing a small disk from $\Sigma$, we obtain a genus $g_{4}(L)$ cobordism between $L$ and the unknot. Then, Proposition \ref{proposition:slice_genus_bound} tells us that
\[- g_{4}(L) =  \nu (\bigcirc) - g(\Sigma) - 1 + 1 \leq \nu(L) \leq g(\Sigma) + \nu (\bigcirc) + \ell - 1 =  g_{4}(L) + \ell -1.\]
Moreover, if $L$ is strongly slice then $L$ is strongly concordant to an unlink. Since $\nu$ is a strong concordance invariant we have that $\nu(L) = \nu(\bigcirc_{\ell}) = 0$, where $\bigcirc_{\ell}$ denotes the unlink with $\ell$ components.
\end{proof}

Another consequence of Proposition \ref{proposition:slice_genus_bound}, together with Proposition \ref{prop:additivity_for_knots}, is the following corollary.

\begin{cor}\label{corollary:restriction}
If $\nu$ is a slice-torus link invariant, then the restriction of $\nu$ to knots is a slice-torus invariant.\qed
\end{cor}

At this point the following question arises naturally.

\begin{question}
Let $\nu$ be a slice-torus invariant. Is there a slice-torus link invariant $\widetilde{\nu}$ whose restriction to knots is $\nu$? If such $\widetilde{\nu}$ exists, is it unique?
\end{question}

The authors believe that all the known slice-torus link invariants admit such an extension, and the answer to the above question is left for future work.

\subsection{Combinatorial bounds and the detection of the slice genus}

Using the slice-genus bound proved in Proposition \ref{proposition:slice_genus_bound}, we can adapt the arguments used by Kawamura (\cite{Kawamura}) for slice-torus invariants, and Abe (\cite{Abe}, see also \cite{Lewark14}) for the Rasmussen invariant $s$, to the case of slice-torus link invariants.

Before going into the details, we need to introduce some notation. Let $D$ be an oriented link diagram representing a link $L$. Denote by $n(D)$, $n_{+}(D)$ and $n_{-}(D)$ the number of crossings, positive crossings and negative crossings of $D$, respectively. 
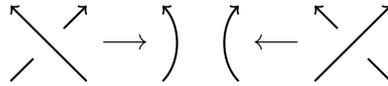
\begin{figure}[H]
\centering
\begin{tikzpicture}[scale = .5]
\draw[thick, ->] (0,0) -- (2,2);
\draw[thick, ->] (10,0) -- (8,2);
\pgfsetlinewidth{40*\pgflinewidth}
\draw[white] (0,2) -- (2,0);
\draw[white] (10,2) -- (8,0);
\pgfsetlinewidth{.025*\pgflinewidth}
\draw[thick,<-] (0,2) -- (2,0);
\draw[thick,<-] (10,2) -- (8,0);

\node at (3,1) {$\longrightarrow$};
\node at (7,1) {$\longleftarrow$};

\draw[thick,<-] (6,2) .. controls +(-.5,-.5) and + (-.5,.5) .. (6,0);
\draw[thick,<-] (4,2) .. controls +(.5,-.5) and + (.5,.5) ..  (4,0);
\end{tikzpicture}
\caption{The oriented resolution of a crossing.}\label{figure:oriented_resolution}
\end{figure} 
The oriented resolution of $D$ is the set of circles (\emph{Seifert circles}) obtained by replacing each crossing with its oriented resolution, as shown in Figure \ref{figure:oriented_resolution}. Denote by $O(D)$ the number of Seifert circles of $D$. If $D$ is non-split then a Seifert surface for the link $L$ can be obtained by considering a disk for each Seifert circle, and for each crossing on the diagram we add a band between the corresponding circles. This procedure goes often under the name of \emph{Seifert algorithm}, and it is easy to see that it produces a connected, compact, oriented surface in $\mathbb{S}^{3}$ bounding the link $L$, which is precisely the definition of Seifert surface (see, for example, \cite[Chapter 5]{Rolfsen}). Furthermore, the genus of the surface $\Sigma_{D}$ obtained via the Seifert algorithm can be easily computed as an exercise, and it turns out that
\[ g(\Sigma_{D}) = 1 +\frac{n(D) - O(D) - \ell}{2}. \]
Now we can take the first step towards our combinatorial bound. More precisely, we prove the following theorem, which asserts that the slice-torus link invariants compute the Seifert genus of positive links, and also the slice genus in the non-split case.
\begin{proof}[Proof of Theorem \ref{teorema:slice_genus_positive}]
We suppose first that $L$ is non-split.
Consider the Seifert surface $\Sigma_{D}$ obtained via the Seifert algorithm from $D$. 
We can remove a small disk from $\Sigma_{D}$, and apply the bound in Proposition \ref{proposition:slice_genus_bound} to obtain the following inequalities
\[\nu(L) \leq g_{4}(L) + \ell - 1 \leq g_{3}(L) + \ell - 1  \leq g(\Sigma_{D}) + \ell - 1  = \frac{n(D) - O(D) + \ell}{2}.\]
Now, starting from $L$ we can apply a sequence of $\ell - 1$ fusion moves as in Figure \ref{figure:positive_fusion_moves} to obtain a positive knot $K$, and a connected genus $0$ cobordism from $K$ to $L$.
\begin{figure}[H]
\centering
\begin{tikzpicture}[scale = .75]

\draw[->](2,0) -- (0,2) ; 
\pgfsetlinewidth{20*\pgflinewidth}
\draw[white,->]  (0,0) -- (2,2) ;
\pgfsetlinewidth{.05*\pgflinewidth}
\draw[->] (0,0) -- (2,2) ;

\draw[thin](.5,.5)  .. controls +(.25,-.25) and +(-.5,-.5) .. (1.75,.25) ;
\pgfsetlinewidth{20*\pgflinewidth}
\draw[white](.25,.25)  .. controls +(.25,-.25) and +(-.5,-.5) .. (1.5,.5) ;
\pgfsetlinewidth{.05*\pgflinewidth}
\draw[thin]  (.25,.25)  .. controls +(.25,-.25) and +(-.5,-.5) .. (1.5,.5) ;
\draw[] (0,0)-- (2,2);

\draw[] (1.25,.75)--  (2,0) ;
\draw[->](2.5,1) -- (3.5,1) ; 
\draw[->] (6,1) .. controls +(0,.5) and +(.5,-.5) .. (4,2); 
\pgfsetlinewidth{20*\pgflinewidth}
\draw[white]  (4,1) .. controls +(0,.5) and +(-.5,-.5) ..(6,2) ;
\pgfsetlinewidth{.05*\pgflinewidth}
\draw[->] (4,1)  .. controls +(0,.5) and +(-.5,-.5) .. (6,2) ;
\draw[] (6,0)  .. controls +(0,.5) and +(0,-.5) .. (4,1) ;
\pgfsetlinewidth{20*\pgflinewidth}
\draw[ white] (4,0)  .. controls +(0,.5) and +(0,-.5) .. (6,1) ;
\pgfsetlinewidth{.05*\pgflinewidth}
\draw[] (4,0)  .. controls +(0,.5) and +(0,-.5) .. (6,1) ;

\node[above left] at (0,0) {$\mathbf{a}$};
\node[above right] at (2,0) {$\mathbf{b}$};
\end{tikzpicture}
\caption{An oriented band move adding a positive crossing to the diagram. Notice that this procedure does not change the number of Seifert circles. Moreover, if the arcs \textbf{a} and \textbf{b} belong to different components then the band move is a fusion move.}\label{figure:positive_fusion_moves}
\end{figure}
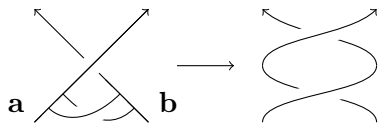

The slice-torus link invariants detect the genus of positive knots (Corollary \ref{corollary:restriction} and \cite[Theorem 4.4]{Kawamura}) and their value can be computed directly from a positive diagram. This leads us to the following sequence of equalities
\[\frac{n(D)  - O(D) + \ell  }{2} =\frac{\overbrace{n(D) + \ell - 1}^{n(D_{K})} - O(D) + 1}{2} = \nu(K),\]
where $D_{K}$ is the positive diagram of $K$ obtained from $D$ by attaching bands as shown in Figure \ref{figure:positive_fusion_moves}.
Finally, from Proposition \ref{proposition:slice_genus_bound} we obtain the inequality
\[\nu(K) \leq \nu(L),\]
and the claim follows.

The statement for split links is proved by observing that all the quantities, except for the slice-genus, involved in the equalities for non-split links are additive under disjoint unions.
\end{proof}

\begin{cor}\label{corollary:slice_genus_pos_torus_links}
Let $\nu$ be a slice-torus link invariant. Then, for each coherently oriented, positive torus link $T$ of type $T_{m,n}$ we have the equalities
\[\nu(T) = g_{3}(T) + \ell - 1 = g_{4}(T) + \ell - 1 = \frac{(n-1)(m-1) + \ell  - 1}{2},\]
where $\ell = GCD(m,n)$ is the number of components of $T$.
\end{cor}
\begin{proof}
The statement follows immediately from Theorem \ref{teorema:slice_genus_positive}. 
\end{proof}

Before proceeding further we need some more notation.
First, we need to describe how to associate to $D$ a graph $\Gamma(D)$, called the \emph{ Seifert graph}. The vertices of $\Gamma(D)$ are the Seifert circles, and there is an edge between two vertices for each crossing the corresponding circles share in $D$. An edge of the Seifert graph is \emph{positive} (resp.  \emph{negative}) if the corresponding crossing is positive (resp. negative). Let $s_{+}(D)$ (resp. $s_{-}(D)$) denote the number of connected components of the graph obtained from $\Gamma(D)$ by removing all the negative (resp. positive) edges.

There is another graph $G(D)$ we can obtain from $D$. This graph has one vertex for each component of the link $L$ represented by $D$, and two vertices of $G(D)$ share a crossing if there is at least a negative crossing joining the corresponding components.

Now that all the notation is set into place, we can state the following lemma which is basically due to Kawamura. For the sake of completeness we will spell out the proof.

\begin{lemma}[\cite{Kawamura}, Lemma 5.5]\label{lemma:5.5}
If $L$ is a link with a non-splittable  diagram $D$, then there exists a positive link $L_{+}$, a diagram $D_{+}$ for $L_{+}$, and a cobordism\footnote{Recall that for us each cobordism between links is such that each connected component of the cobordism has boundary touching both links.} $\Sigma_{+}$ from $L$ to $L_{+}$ such that
\[ n(D_{+}) = n_{+}(D) + s_{+}(D) - 1,\quad O(D_{+}) = O(D) \]
and 
\[\chi(\Sigma_{+}) = -n_{-}(D) + s_{+}(D) - 1.\]
Furthermore, the number of components of $\Sigma_{+}$ is lower than or equal to the number of components of $G(D)$.
\end{lemma}
\begin{proof}
Let us start from $D$, via band moves (cf. the first part of Figure \ref{figure:crossing_change}) we can eliminate all the negative crossings in $D$. With this procedure we end up with a collection of $s_{+}(D)$ link diagrams, say $D_{1}$, ..., $D_{s_{+}(D)}$. Moreover, we can see the Seifert graph of each $D_{i}$ as a sub-graph of $\Gamma(D)$. Consider the graph $\overline{\Gamma}$ obtained from $\Gamma(D)$ by collapsing each $\Gamma(D_i)$. Notice that each vertex of $\overline{\Gamma}$ corresponds to a $D_{i}$. Since $L$ is non-split, the diagram $D$ is connected (as a graph) and non-splittable. In particular, $\Gamma(D)$ (and thus $\overline{\Gamma}$) is connected.
Pick a spanning tree $T$ for $\overline{\Gamma}$. Via band moves we add a positive crossing (cf. the second part of Figure \ref{figure:crossing_change}) between $D_{i}$ and $D_{j}$ if the corresponding vertices in $\overline{\Gamma}$ are joined by an edge in $T$. Call $D_{+}$ and $\Sigma_{+}$ the diagram and the surface, respectively, obtained via the procedure just described.

The computation of the number of crossings and the number of Seifert circles of $D_+$, and the computation of the Euler characteristic of $\Sigma_{+}$ are easily done. Moreover, $\Sigma_{+}$ is a cobordism by construction.

All that is left is to count the number of connected component of $\Sigma_+$. Since $\Sigma_{+}$ is a cobordism each connected components touches a component of $L$. Moreover, if there is a negative crossing in $D$ between two components $L_1$ and $L_2$, then there is a band joining them. It follows that $L_1$ and $L_2$ belong to the same connected component of $\Sigma_+$.
\end{proof}
\begin{proof}[Proof of Theorem \ref{teorema:bound_combinatorio}]
Let us borrow the notation from the statement of Lemma \ref{lemma:5.5}. First, we wish to compute the genus of $\Sigma_{+}$. We may assume $L$ non-split since all the quantities in the statement are additive under disjoint union. From the general formula
\[\chi(\Sigma) = 2c(\Sigma) - 2 g(\Sigma) - c(\partial \Sigma),\]
where $c$ denotes the number of connected components, we obtain
\[ - g(\Sigma) = \frac{\chi(\Sigma) +c(\partial\Sigma)}{2} - c(\Sigma).\]
Denoted by $\ell_+$ the number of components of $L_+$, plugging in $\Sigma_{+}$ and replacing the corresponding quantities with their value we obtain
\[-g(\Sigma_{+}) = \frac{-n_{-}(D) + s_{+}(D) - 1 + \ell + \ell_{+}}{2} - c(\Sigma_{+}).\]
We should argue that we may assume $G(D)$ to be connected. This is easily done by replacing $D$ with a diagram $D^\prime$ such that:
\[w(D) = w(D^\prime),\quad O(D) = O(D^\prime),\quad s_{+}(D) = s_{+}(D^\prime),\]
and $c(G(D^\prime)) =1$. This can be obtained by choosing a positive crossing between each pair of components of $L$ which share only positive crossings, and perform a second Reidemeister as illustrated in Figure \ref{figura:trucco_seconda_mossa}. Since, $ c(\Sigma_{+})\leq c(G(D)) = 1$, we may assume $\Sigma_{+}$ to be connected.
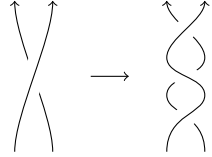
\begin{figure}[h]
\centering
\begin{tikzpicture}[scale = .5]

\draw[->] (1,-1) .. controls +(0,1) and +(0,-.5) .. (0,.5).. controls +(0,.5) and +(0,-.5) .. (1,1.5).. controls +(0,.5) and +(0,-.5) .. (0,3);
\pgfsetlinewidth{20*\pgflinewidth}
\draw[white] (0,-1) .. controls +(0,1) and +(0,-.5) .. (1,.5).. controls +(0,.5) and +(0,-.5) .. (0,1.5).. controls +(0,.5) and +(0,-.5) .. (1,3);
\pgfsetlinewidth{0.05*\pgflinewidth}
\draw[->] (0,-1) .. controls +(0,1) and +(0,-.5) .. (1,.5).. controls +(0,.5) and +(0,-.5) .. (0,1.5).. controls +(0,.5) and +(0,-.5) .. (1,3);

\draw[->] (-2,1) -- (-1,1);

\draw[->] (-3,-1)  .. controls +(0,1) and +(0,-1) .. (-4,3);
\pgfsetlinewidth{20*\pgflinewidth}
\draw[white] (-4,-1)  .. controls +(0,1) and +(0,-1) .. (-3,3);
\pgfsetlinewidth{0.05*\pgflinewidth}
\draw[->] (-4,-1) .. controls +(0,1) and +(0,-1) .. (-3,3);
\end{tikzpicture}
\caption{A second Reidemeister move near a positive crossing.}\label{figura:trucco_seconda_mossa}
\end{figure}

Now, consider the quantity
\[\nu(L_{+}) - g(\Sigma_+) -\ell_+ + 1 =\]
since $L_{+}$ is positive, from Theorem \ref{teorema:slice_genus_positive} it follows that
\[ = \frac{n_{+}(D) + s_{+}(D) - 1 - O(D) + \ell_+}{2} + \frac{-n_{-}(D) + s_{+}(D) - 1 + \ell + \ell_{+}}{2} - 1 - \ell_+ + 1 =\]
and simple computations show that
\[ = \frac{w(D) -O(D) + 2 s_{+}(D) + \ell - 2}{2}. \]
Finally, Proposition \ref{proposition:slice_genus_bound} tells us that
\[ \nu(L_{+}) - g(\Sigma_+) -\ell_+ + 1 \leq \nu(L),\]
and the result follows.

\end{proof}

The combinatorial bound presented in Theorem \ref{teorema:bound_combinatorio} is analogous to the bounds presented in \cite{Cavallo4, Kawamura} (see also \cite{Abe, Lewark14, Lobb11}) for the Rasmussen and Rasmussen-Beliakova-Wehrli invariants. A possible direction of work might be to find an analogue of the combinatorial bound presented in \cite{Collari}. Let us leave this matter aside for now, and 
let us turn to the last result of this section.
\begin{prop}\label{proposition:slice_genus_negative}
Let $L$ be a negative link, and let $\ell_s$ be the number of its split components. Then
\[\nu(L) = \frac{-n(D) + O(D) + \ell -2 \ell_{s}}{2},\]
for each negative diagram $D$ and each slice-torus link invariant $\nu$.
\end{prop}
\begin{proof}
Combining Corollary \ref{corollary:restriction} with \cite[Theorem 5]{Lewark14} (notice the different normalization, and see also \cite{Abe}), we obtain that our claim is true for negative knots.
Since the quantities $\nu$, $n$, $O$, $\ell$ and $\ell_{s}$ are additive with respect to the disjoint union, we may assume $L$ to be non-split ($\ell_s = 1$).
The proof goes by induction on the number of components of $L$.
Suppose the claim true for all $1  \leq \ell < r$, and assume $\ell = r$. Then by performing a band move, similar to the one in Figure \ref{figure:positive_fusion_moves}, between two components we can obtain a negative link $L^\prime$ which has $\ell-1$ components, is non-split and has a negative diagram with $n(D) + 1$ crossings and $O(D)$ Seifert circles. By Property (A) we have
\[ \nu(L) \leq \nu(L^\prime) + 1 = \frac{-n(D) - 1 + O(D) + \ell -1 -2}{2} + 1 = \frac{-n(D)+ O(D) + \ell -2}{2},\]
where the first equality is the inductive hypothesis. 
Since any negative diagram is non-splittable, the other inequality follows from Theorem \ref{teorema:bound_combinatorio}.
\end{proof}
This proposition allows us to prove the equivalent of Corollary \ref{corollary:slice_genus_pos_torus_links} for coherently oriented, negative torus links.
\begin{cor}
 Let $T^*$ be the mirror image of a torus link of type $T_{m,n}$ with all the components oriented in the same direction.
 Then for every slice-torus link invariant $\nu$ we have the equalities
 \[ \nu(T^*)=-g_3(T)=-g_4(T)=\dfrac{\ell-1-(n-1)(m-1)}{2}\:,\]
 where $\ell=\text{GCD}(m,n)$ is the number of components of $T$.
\end{cor}
\begin{proof}
 Proposition \ref{proposition:slice_genus_negative} says that \[\nu(L^*)=\ell-1-\nu(L)\] if $L$ is a non-split positive link. Then the claim follows from this observation, Corollary \ref{corollary:slice_genus_pos_torus_links} and the fact that $g_{4}(L) = g_{4}(L^{*})$ and $g_{3}(L) = g_{3}(L^{*})$.
\end{proof}

\section{Applications}

This section is dedicated to two applications. The first is an application of the combinatorial bound, and consist of the computation of the slice-torus link invariants of quasi-positive links. The second application is a lower bound on the splitting number of links.

\subsection{Quasi-positive links}
Let us recall the definition of quasi-positive braid and quasi-positive link.
\begin{defin}
 A quasi-positive link is any link which can be realized as the closure 
 of a $d$-braid of the
 form \[\prod_{i=1}^bw_i\sigma_{j_i}w_i^{-1}\:,\] where $\sigma_j$ for $j=1,...,d-1$ are the Artin generators of the $d$-braids group.
\end{defin} 
Thus quasi-positive links are closures of braids consisting of arbitrary conjugates of positive (Artin) generators.

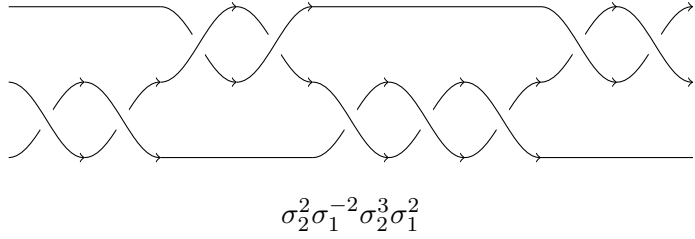
\begin{figure}[h]
\centering
\begin{tikzpicture}[scale = .5]

\draw[<-] (3.5,2) .. controls +(-.75,0) and +(.75,0) .. (1.5,0);
\pgfsetlinewidth{20*\pgflinewidth}
\draw[white] (1.5,2) .. controls +(.75,0) and +(-.75,0) .. (3.5,0);
\pgfsetlinewidth{.05*\pgflinewidth}
\draw[->] (1.5,2) .. controls +(.75,0) and +(-.75,0) .. (3.5,0);
\begin{scope}[shift = {+(2,0)}]
\draw[<-] (3.5,2) .. controls +(-.75,0) and +(.75,0) .. (1.5,0);
\pgfsetlinewidth{20*\pgflinewidth}
\draw[white] (1.5,2) .. controls +(.75,0) and +(-.75,0) .. (3.5,0);
\pgfsetlinewidth{.05*\pgflinewidth}
\draw[->] (1.5,2) .. controls +(.75,0) and +(-.75,0) .. (3.5,0);
\end{scope}

\begin{scope}[shift = {+(4,2)}]
\draw[->] (1.5,2) .. controls +(.75,0) and +(-.75,0) .. (3.5,0);
\pgfsetlinewidth{20*\pgflinewidth}
\draw[white] (3.5,2) .. controls +(-.75,0) and +(.75,0) .. (1.5,0);
\pgfsetlinewidth{.05*\pgflinewidth}
\draw[<-] (3.5,2) .. controls +(-.75,0) and +(.75,0) .. (1.5,0);
\end{scope}
\begin{scope}[shift = {+(6,2)}]
\draw[->] (1.5,2) .. controls +(.75,0) and +(-.75,0) .. (3.5,0);
\pgfsetlinewidth{20*\pgflinewidth}
\draw[white] (3.5,2) .. controls +(-.75,0) and +(.75,0) .. (1.5,0);
\pgfsetlinewidth{.05*\pgflinewidth}
\draw[<-] (3.5,2) .. controls +(-.75,0) and +(.75,0) .. (1.5,0);
\end{scope}
\begin{scope}[shift = {+(12,0)}]
\draw[<-] (3.5,2) .. controls +(-.75,0) and +(.75,0) .. (1.5,0);
\pgfsetlinewidth{20*\pgflinewidth}
\draw[white] (1.5,2) .. controls +(.75,0) and +(-.75,0) .. (3.5,0);
\pgfsetlinewidth{.05*\pgflinewidth}
\draw[->] (1.5,2) .. controls +(.75,0) and +(-.75,0) .. (3.5,0);
\end{scope}
\begin{scope}[shift = {+(8,0)}]
\draw[<-] (3.5,2) .. controls +(-.75,0) and +(.75,0) .. (1.5,0);
\pgfsetlinewidth{20*\pgflinewidth}
\draw[white] (1.5,2) .. controls +(.75,0) and +(-.75,0) .. (3.5,0);
\pgfsetlinewidth{.05*\pgflinewidth}
\draw[->] (1.5,2) .. controls +(.75,0) and +(-.75,0) .. (3.5,0);
\end{scope}
\begin{scope}[shift = {+(10,0)}]
\draw[<-] (3.5,2) .. controls +(-.75,0) and +(.75,0) .. (1.5,0);
\pgfsetlinewidth{20*\pgflinewidth}
\draw[white] (1.5,2) .. controls +(.75,0) and +(-.75,0) .. (3.5,0);
\pgfsetlinewidth{.05*\pgflinewidth}
\draw[->] (1.5,2) .. controls +(.75,0) and +(-.75,0) .. (3.5,0);
\end{scope}
\begin{scope}[shift = {+(16,2)}]
\draw[<-] (3.5,2) .. controls +(-.75,0) and +(.75,0) .. (1.5,0);
\pgfsetlinewidth{20*\pgflinewidth}
\draw[white] (1.5,2) .. controls +(.75,0) and +(-.75,0) .. (3.5,0);
\pgfsetlinewidth{.05*\pgflinewidth}
\draw[->] (1.5,2) .. controls +(.75,0) and +(-.75,0) .. (3.5,0);
\end{scope}
\begin{scope}[shift = {+(14,2)}]
\draw[<-] (3.5,2) .. controls +(-.75,0) and +(.75,0) .. (1.5,0);
\pgfsetlinewidth{20*\pgflinewidth}
\draw[white] (1.5,2) .. controls +(.75,0) and +(-.75,0) .. (3.5,0);
\pgfsetlinewidth{.05*\pgflinewidth}
\draw[->] (1.5,2) .. controls +(.75,0) and +(-.75,0) .. (3.5,0);
\end{scope}
\draw (5.5,4) -- (1.5,4);
\draw (5.5,0) -- (9.5,0);
\draw (9.5,4) -- (15.5,4);
\draw (15.5,0) -- (19.5,0);
\node at (10.5,-1.5) {$\sigma_{2}^{2}\sigma_{1}^{-2}\sigma_{2}^{3}\sigma_{1}^{2}$};
\end{tikzpicture}
\caption{A quasi-positive 3-braid and its ``geometrical'' representation.}
\end{figure}

\begin{teo}\label{teo:quasi-pos}
Consider the $d$-braid $B=(w_1\sigma_{j_1}w_1^{-1})\cdot...\cdot(w_b\sigma_{j_b}w_b^{-1})$, and denote by $L$ its closure.
Then, for every slice-torus link invariant $\nu$ we have the equality
 \[\nu(L)=\dfrac{b-d+\ell}{2}\:,\]
 where $\ell$ is the number of components of $L$.
\end{teo}
\begin{proof}
Since all quantities involved in the statement are additive under disjoint union, we may assume $L$ to be non-split. First, we wish to prove the inequality
\[\nu(L)\leq\dfrac{\ell-\chi(\Sigma)}{2}\:,\]
where $\Sigma$ is a compact oriented surface, properly embedded in $D^4$, such that $\partial\Sigma=L$.
Assume $\Sigma$ has $k$ connected components. Then Proposition \ref{proposition:slice_genus_bound} tells us that 
\[\nu(L)\leq g(\Sigma)+\ell-k=\dfrac{\ell-\chi(\Sigma)}{2}\:.\]

Since $L$ bounds a surface $\Sigma_B$ which satisfies the previous properties and is such that $\chi(\Sigma_B)=d-b$, as it is shown in \cite{Rudolph3}, we obtain that
\[\nu(L)\leq\dfrac{b-d+\ell}{2}\:.\]
 
The other inequality follows from the bound in Equation \eqref{equation:bound_combiantorio}. In fact, this gives that
\[\dfrac{b-d+2s_+(B)+\ell-2}{2}=\dfrac{w(B)-O(B)+2s_+(B)+\ell-2}{2}\leq\nu(L).\] 
Since $1 \leq s_{+}(B)$ the statement follows.
\end{proof}

\subsection{Splitting number}
As we anticipated, the slice-torus link invariants can be used to obtain a lower bound for the splitting number $\widetilde{\spl}$ of a link (\cite{Adams}), which is sometimes called weak splitting number (\cite{FP}). Let us recall its definition first.
\begin{defin}
The \emph{splitting number} $\widetilde{\spl}(L)$ of a link $L$ is defined as the minimum number of crossing changes to perform on a diagram (for all possible diagrams) of $L$ in order to turn the link into a disjoint union of knots.
\end{defin}
Note that in literature the symbol $\spl(L)$ usually denotes a different version of the splitting number of $L$, which we called the strong splitting number in the introduction. The \emph{strong splitting number} is defined exactly as $\widetilde{\spl}$ but the only crossing changes allowed are those between different components. In particular, we have that $\widetilde{\spl}(L)\leq\spl(L)$.

\begin{remark}\label{remark:linkingandsplitting}
For each oriented link $L$, we have
\[ \spl(L) \equiv \sum_{1\leq i<j \leq \ell} lk(L_{i},L_{j})\quad \mod (2),\]
where $L_{1}$,...,$L_{\ell}$ denote the components of $L$. This fact can be easily proved by induction, alternatively the reader can consult \cite[Lemma 2.1]{Friedletal}.
\end{remark}

\begin{proof}[Proof of Theorem \ref{theorem:splitting}]
If $\widetilde{\spl}(L)=0$, then $L$ is a disjoint union of knots. The additivity of $\nu$ (Property (B)) tells us that in this case the left hand side of Equation \eqref{equation:bound_on_splitting} is also zero. Thus, the (in)equality holds. We claim that the quantity
\[\left|\nu(L)-\sum_{i=1}^{\ell}\nu(K_i)\right|\]
increases at most by $1$ at each crossing change. The result is proved by induction on the value of $\widetilde{\spl}(L)$ as follows; consider a minimal sequence of crossing changes from $L$ to a split union of knots. Denote by $L^\prime$ the first step in this sequence, then
\[\left|\nu(L)-\sum_{i=1}^{\ell}\nu(K_i)\right|  \leq  \left|\nu(L^\prime)-\sum_{i=1}^{\ell}\nu(K^\prime_i)\right| + 1\leq \widetilde{\spl}(L^\prime)  + 1= \widetilde{\spl}(L) , \]
where the first inequality is our claim, and the second inequality follows from the inductive hypothesis.

Now, let us prove our claim. First, assume the crossing change to happen between different components. In particular, none of the $K_i$'s is modified under this crossing change, while $\nu(L)$ can either increase or decrease at most by $1$ (cf. Proposition \ref{proposition:crossing_change}).
Now, assume the crossing change to be performed on a component of $L$, say $K = K_{i}$ for some $i$. This crossing change modifies both $L$ and $K$, but leaves all the other components unchanged.
Again from Proposition \ref{proposition:crossing_change} it follows that
 \begin{equation}
 \label{equation:cross_chg1}
 \nu(L_+)-1\leq\nu(L_-)\leq\nu(L_+)\:,    
 \end{equation}
 and that
 \begin{equation}
  \label{equation:cross_chg2}
  -\nu(K_{+})\leq-\nu(K_{-})\leq-\nu(K_{+})+1\:,
 \end{equation}
 where the plus and minus denote the signs of the crossing, before and after the change.
Adding Equation \eqref{equation:cross_chg1} and Equation \eqref{equation:cross_chg2}, we obtain
 \[\nu(L_+)-\nu(K_{+})-1\leq\nu(L_-)-\nu(K_{-})\leq\nu(L_+)-\nu(K_{+})+1.\]
Since either $L=L_+$ and $K=K_+$, or $L=L_-$ and $K=K_-$, and all the other components of $L$ are left unchanged, the claim follows.
\end{proof}

\begin{cor}
Let $p$ and $q$ be coprime integers, and $k > 0$. Then, the following equality holds
\[ \spl (T_{kp,kq}) = \widetilde{\spl}(T_{kp,kq}) = \frac{ k (  k  - 1)\vert pq \vert}{2}\]
\end{cor}
\begin{proof}
 In \cite[Corollary 3]{Jeong}, the author proves that
\begin{equation}
\label{eq:splfortorus}
\left|\nu_{s_n}(T)-\sum_{i=1}^{k}\nu_{s_{n}}(K_i)\right| = \spl (T_{kp,kq}) = \frac{ k (  k  - 1) pq }{2}.
\end{equation}
where $p~,q > 0$, $T$ is the positive torus link of type $T_{kp,kq}$, and $K_1$,..., $K_{k}$ are the components of $T$.
Notice that by Corollary~\ref{corollary:slice_genus_pos_torus_links} the value of any slice-torus link invariant on positive torus links does not depend on the chosen invariant, therefore Jeong's computation holds for any slice-torus link invariant. From Equation \eqref{eq:splfortorus} and from
\begin{equation}
 \label{equation:J}
 \left|\nu(L)-\sum_{i=1}^{\ell}\nu(K_i)\right|\leq \widetilde{\spl}(L)\leq\spl(L),   
\end{equation}
the desired equality follows for $p~,q > 0$. In the other cases, at most we recover the mirror image of $T_{k\vert p\vert ,k\vert q\vert }$, and since the splitting number of a link and its mirror is the same, the corollary follows.
\end{proof}

\begin{remark}
Notice that $\widetilde{\spl}(L)$ does not depend on the orientation of $L$. It follows that the inequality in Theorem \ref{theorem:splitting} holds for every relative orientation of the link. Thus, the maximum among all these values is still a lower bound for $\widetilde{\spl}(L)$.
\end{remark}

In \cite{Adams,FP} some lower bounds for $\widetilde{\spl}$ are also given. In this paper we describe an infinite family of 2-components links for which Theorem \ref{theorem:splitting} allows us to compute $\widetilde{\spl}$, where all the obstructions in \cite{Adams,FP} fail.

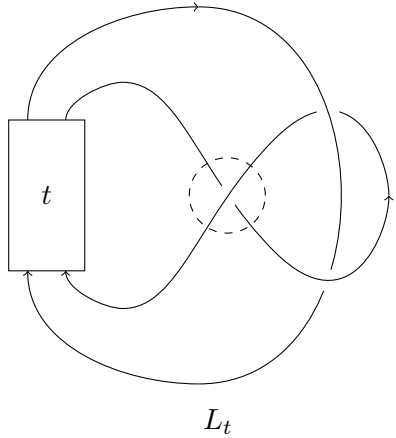
\begin{figure}[h]
\centering
\begin{tikzpicture}[scale =.5]

\draw[->] (-.5,2) .. controls +(0,2) and +(-2,0) .. (4,5);
\draw[<-] (-.5,-2) .. controls +(0,-2) and +(-2,0) .. (4,-5);

\draw[] (.5,2)  .. controls +(0,.5) and +(-.5,0) .. (2,3) .. controls +(1.5,0) and +(-2,2) .. (6,-1.5);
\pgfsetlinewidth{20*\pgflinewidth}
\draw[white] (.5,-2) .. controls +(0,-.5) and +(-.5,0) .. (2,-3) .. controls +(1.5,0) and +(-2,-2) .. (6,1.5);
\pgfsetlinewidth{0.05*\pgflinewidth}
\draw[<-] (.5,-2) .. controls +(0,-.5) and +(-.5,0) .. (2,-3) .. controls +(1.5,0) and +(-2,-2) .. (6,1.5);

\draw[] (9,0)  .. controls +(0,1) and +(2,2) .. (6,1.5);
\pgfsetlinewidth{20*\pgflinewidth}
\draw[white] (4,-5) .. controls +(5,0) and +(5,0) .. (4,5);
\pgfsetlinewidth{0.05*\pgflinewidth}
\draw[] (4,-5) .. controls +(5,0) and +(5,0) .. (4,5);

\pgfsetlinewidth{20*\pgflinewidth}
\draw[white] (9,0) .. controls +(0,-1) and +(2,-2) .. (6,-1.5);
\pgfsetlinewidth{0.05*\pgflinewidth}
\draw[<-] (9,0) .. controls +(0,-1) and +(2,-2) .. (6,-1.5);

\draw (-1,2) rectangle (1,-2);
\node at (0,0) {$t$};
\draw[dashed] (4.75,0) circle (1);
\node at (4.5,-6) {$L_t$};

\end{tikzpicture}
\caption{A diagram for the 2-component link $L_t$, where $t\in \mathbb{N}\setminus\{0\}$ denotes the number of positive full twists.}\label{figure:L_t}
\end{figure}

\begin{prop}
 \label{proposition:splitting}
 Let us consider the links $L_t$ in Figure \ref{figure:L_t}. Then, we have that $\widetilde{\spl}(L_t)=t$ and $\spl(L_{t}) = t+1$ for every $t\geq3$.
\end{prop}
\begin{proof}
 We use the link version of the $\tau$-invariant, see \cite{Cavallo}. Since $L_t$ is non-split alternating for every
 $t\geq1$, we have that $\tau(L_t)$ is determined by the signature:
\[\tau(L_t)=\dfrac{\ell-1-\sigma(L_t)}{2}=\dfrac{1-\sigma(L_t)}{2}\:;\] where $\ell$ is the number of components of the link, which is always equal to two in this case.

An easy computation gives that $\sigma(L_t)=1-2t$, and thus $\tau(L_t)=t$, for every $t\geq1$. The link $L_t$ has unknotted components, hence Theorem \ref{theorem:splitting} implies $t\leq\widetilde{\spl}(L_t)$. On the other hand, we immediately see that $L_t$ can be unlinked (and thus reduced to the split union of knots) by changing $t$ crossings: one crossing for each full twist except one, plus the crossing circled in Figure \ref{figure:L_t}. Thus, we proved that $\widetilde{\spl}(L_{t}) = t$.

It follows that
\[ \spl(L_{t})\geq \widetilde{\spl}(L_{t}) = t.\]
Moreover, a simple computation shows that 
\[\vert lk(K_{t}, K_{t}^{\prime}) \vert = t-1,\]
where $K_{t}$ and $K_t^\prime$ are the components of $L_{t}$. Therefore, by Remark \ref{remark:linkingandsplitting}, we have that $\spl(L_{t})$ is at least $t+1$. Finally, a direct inspection of the diagram in Figure \ref{figure:L_t} tells us that $t+1 \geq \spl(L_{t})$, and the statement follows.
\end{proof}

\section{Whitehead doubles and a concordance invariant for links}

In this section we define some link invariants, related to Livingston and Naik's invariant $t_{\nu}$, and study some of their properties. We start by defining the fully clasped and the reduced Whitehead doubles. Then, we define the functions $F_{\nu},\: F^\prime_{\nu},\: \overline{F}_{\nu}$ and $\overline{F}_{\nu}$ and prove their basic properties. Finally, we prove an obstruction for a link to be concordant to a split link.
We recall that each strong concordance defines a bijection between the components of the two links, identifying them. Throughout this section all links are oriented.

\subsection{Whitehead doubles of links}

Unlike the case of knots, the Whitehead double of links is not uniquely defined. 
\begin{figure}[H]
\centering
\begin{tikzpicture}[scale = .75]
\draw[red] (.6,0) arc (0:180:.6 and .55);
\draw[red, dashed] (.6,0) arc (0:-180:.6 and .55);
\draw[thick] (0,0) circle (2 and 2.1);
\draw[thick] (.6,.1) arc (0:-180:.6 and .5);
\draw[thick] (.55,-.1) arc (0:180:.55);

\draw (-.25,1.225)  .. controls +(0,.05) and +(-.5,0) .. (.25,1.625);
\draw  (-.25,.825) .. controls +(.5,0) and +(0,-.05) .. (.25,1.225) ;
\pgfsetlinewidth{10*\pgflinewidth}
\draw[white]  (.25,1.225)  .. controls +(0,.05) and +(.5,0) .. (-.25,1.625);
\draw[white]  (.25,.825) .. controls +(-.5,0) and +(0,-.05) .. (-.25,1.225);
\pgfsetlinewidth{.1*\pgflinewidth}
\draw (.25,1.225)  .. controls +(0,.05) and +(.5,0) .. (-.25,1.625);
\draw (.25,.825) .. controls +(-.5,0) and +(0,-.05) .. (-.25,1.225);

\draw[->] (.25,.825) -- (.5,.825)  .. controls +(.75,0) and +(.75,0) .. (.75,-.825) ;
\draw[<-] (.25,1.625) -- (.5,1.625)  .. controls +(1.5,0) and +(1.25,0) .. (.75,-1.625) ;
\draw[<-] (-.25,.825) -- (-.5,.825)  .. controls +(-.75,0) and +(-.75,0) .. (-.75,-.825) ;
\draw[->] (-.25,1.625) -- (-.5,1.625)  .. controls +(-1.5,0) and +(-1.25,0) .. (-.75,-1.625) ;
\draw (-.75,-.75) rectangle (.75,-1.75);
\node at (0,-1.25) {$t$};

\node at (0,-3) {$W^{+}_{t}$};

\begin{scope}[shift = {+(6,0)}]
\draw[red] (.6,0) arc (0:180:.6 and .55);
\draw[red, dashed] (.6,0) arc (0:-180:.6 and .55);
\draw[thick] (0,0) circle (2 and 2.1);
\draw[thick] (.6,.1) arc (0:-180:.6 and .5);
\draw[thick] (.55,-.1) arc (0:180:.55);

\draw (.25,1.225)  .. controls +(0,.05) and +(.5,0) .. (-.25,1.625);
\draw (.25,.825) .. controls +(-.5,0) and +(0,-.05) .. (-.25,1.225);
\pgfsetlinewidth{10*\pgflinewidth}
\draw[white] (-.25,1.225)  .. controls +(0,.05) and +(-.5,0) .. (.25,1.625);
\draw[white]  (-.25,.825) .. controls +(.5,0) and +(0,-.05) .. (.25,1.225) ;
\pgfsetlinewidth{.1*\pgflinewidth}
\draw (-.25,1.225)  .. controls +(0,.05) and +(-.5,0) .. (.25,1.625);
\draw  (-.25,.825) .. controls +(.5,0) and +(0,-.05) .. (.25,1.225) ;

\draw[->] (.25,.825) -- (.5,.825)  .. controls +(.75,0) and +(.75,0) .. (.75,-.825) ;
\draw[<-] (.25,1.625) -- (.5,1.625)  .. controls +(1.5,0) and +(1.25,0) .. (.75,-1.625) ;
\draw[<-] (-.25,.825) -- (-.5,.825)  .. controls +(-.75,0) and +(-.75,0) .. (-.75,-.825) ;
\draw[->] (-.25,1.625) -- (-.5,1.625)  .. controls +(-1.5,0) and +(-1.25,0) .. (-.75,-1.625) ;
\draw (-.75,-.75) rectangle (.75,-1.75);
\node at (0,-1.25) {$t$};
\node at (0,-3) {$W^{-}_{t}$};
\end{scope}


\draw (11.5,.75)  .. controls +(.75,0) and +(-.75,0) .. (13,1.75) ;
\draw (10,.75)  .. controls +(.75,0) and +(-.75,0) .. (11.5,1.75) ;
\pgfsetlinewidth{10*\pgflinewidth}
\draw[white]  (10,1.75)  .. controls +(.75,0) and +(-.75,0) .. (11.5,.75);
\draw[white]  (11.5,1.75)  .. controls +(.75,0) and +(-.75,0) .. (13,.75) ;
\pgfsetlinewidth{.1*\pgflinewidth}
\draw[<-] (10,1.75)  .. controls +(.75,0) and +(-.75,0) .. (11.5,.75);
\draw[->] (11.5,1.75)  .. controls +(.75,0) and +(-.75,0) .. (13,.75) ;
\node at (11.5,0) {a positive full twist};

\draw[<-] (10,-1.25)  .. controls +(.75,0) and +(-.75,0) .. (11.5,-2.25);
\draw[->] (11.5,-1.25)  .. controls +(.75,0) and +(-.75,0) .. (13,-2.25) ;

\pgfsetlinewidth{10*\pgflinewidth}
\draw[white]  (11.5,-2.25)  .. controls +(.75,0) and +(-.75,0) .. (13,-1.25) ;
\draw[white]  (10,-2.25)  .. controls +(.75,0) and +(-.75,0) .. (11.5,-1.25) ;
\pgfsetlinewidth{.1*\pgflinewidth}
\draw (11.5,-2.25)  .. controls +(.75,0) and +(-.75,0) .. (13,-1.25) ;
\draw (10,-2.25)  .. controls +(.75,0) and +(-.75,0) .. (11.5,-1.25) ;
\node at (11.5,-3) {a negative full twist};
\end{tikzpicture}
\caption{The patterns $W^{\pm}_{t}$. The box represents either $\vert t \vert$ positive full twists or $\vert t \vert$ negative full twists, depending on whether $t$ is positive or negative.}\label{figura:Whiteheadknot}
\end{figure}
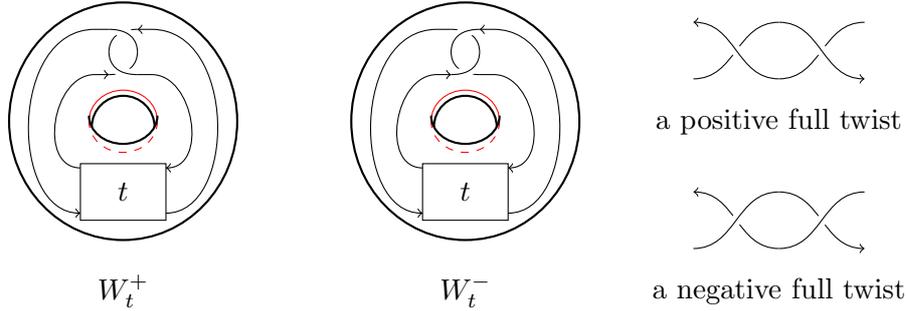
In this paper we use two among the possible definitions of Whitehead double. The two constructions give non-isotopic links, unless our link is a knot or the unlink.

The first family we introduce shall be referred to as fully clasped Whitehead doubles, and is defined as follows. Let $L$ be a link with $\ell$ components, and let $\underline{t}= (t_1,...,t_\ell)\in\mathbb{Z}^\ell $. The \emph{positively} (resp. \emph{negatively}) \emph{fully clasped Whitehead double} $W_{\pm}(L,\underline{t})$ is the $\ell$-component link obtained by the satellite of companion $L$, with pattern\footnote{The homeomorphism sending the torus containing the knot $W_{t}^\pm$ to a tubular neighbouhood of each component is assumed to send the longitude drawn in red in Figure  \ref{figura:Whiteheadknot} to the longitude determined by a Seifert surface.} on the $i$-th component given by the positively (resp. negatively) clasped $t$-twist knot $W^\pm_{t_{i}}$ (see Figure \ref{figura:Whiteheadknot}).

The second family of Whitehead doubles considered in this paper is given by the reduced Whitehead doubles. Let $L$ be a link. Fix $t\in\mathbb{Z}$. The \emph{positive} (resp. \emph{negative}) \emph{reduced Whitehead double} $W^\prime_{\pm}(L,t;L_1)$ is the $\ell$-component link obtained by the satellite of companion $L$, with pattern on $L_1$ given by the positively (resp. negatively) clasped twist knot $W^\pm_{t}$ (Figure \ref{figura:Whiteheadknot}).

For both these families there are diagrams which can be easily described directly from a diagram $D$ of $L$. Given a diagram $D$, representing $L$, denote by $D_{1}$, ...,$D_{\ell}$ the sub-diagrams representing the components $L_{1}$, ...,$L_{\ell}$. Draw a parallel copy of the diagrams $D_{1}$, ...,$D_{\ell}$, add $t_{i} -w(D_{i})$ full twists between the two copies of $D_{i}$, and insert the clasps in all the components to obtain the diagram  $D_{\pm}(L,\underline{t})$ for the fully clasped Whitehead double. The diagram $D^\prime_{\pm}(L,t;L_{1})$ for the reduced Whitehead double can be obtained as follows: draw a parallel copy of $D_{1}$ (the component corresponding to $L_1$), add $t -w(D_{1})$ full twists and a clasp between the two copies of $D_{1}$, and leave all the other components untouched. An example of such diagrams is depicted in Figure \ref{figure:example_of_whitehead_d}.
\begin{figure}[h]
\centering
\begin{tikzpicture}[scale = .75]

\begin{scope}[shift = {+(-4.5,3.5)}]
\draw (-1,0) arc (180:0:1.5);
\draw (1,0) arc (-180:0:1.5);
\pgfsetlinewidth{20*\pgflinewidth}
\draw[white] (1,0) arc (180:0:1.5);
\draw[white] (-1,0) arc (-180:0:1.5);
\pgfsetlinewidth{.05*\pgflinewidth}
\draw (-1,0) arc (-180:0:1.5);
\draw (1,0) arc (180:0:1.5);
\end{scope}

\draw (-1.5,0) arc (180:0:2.5);
\draw (-1,0) arc (180:0:2);
\pgfsetlinewidth{20*\pgflinewidth}
\draw[white]  (1,0) arc (180:90:2.5);
\draw[white] (1.5,0) arc (180:90:2);
\pgfsetlinewidth{.05*\pgflinewidth}

\draw (1,0) arc (180:0:2.5);
\draw (1.5,0) arc (180:0:2);
\draw (1,0) arc (-180:0:2.5);
\draw (1.5,0) arc (-180:0:2);

\pgfsetlinewidth{20*\pgflinewidth}
\draw[white]  (-1.5,0) arc (-180:0:2.5);
\draw[white] (-1,0) arc (-180:0:2);
\pgfsetlinewidth{.05*\pgflinewidth}

\draw (-1.5,0) arc (-180:0:2.5);
\draw (-1,0) arc (-180:0:2);

\draw[fill, white] (-.75,-.5) rectangle (-1.75,.5);

\draw (-.935,-.51) .. controls +(-.005,.05) and +(.25,0) .. (-1.25,.25);
\pgfsetlinewidth{20*\pgflinewidth}
\draw[white]   (-.935,.51) .. controls +(-.005,.05) and +(.25,0) .. (-1.25,-.25);
\pgfsetlinewidth{.05*\pgflinewidth}
\draw (-.865,.71) .. controls +(-.005,.05) and +(.25,0) .. (-1.25,-.25);

\draw (-1.45,.51) .. controls +(-.005,.05) and +(-.25,0) .. (-1.25,-.25);
\pgfsetlinewidth{20*\pgflinewidth}
\draw[white]   (-1.45,-.51) .. controls +(.005,.05) and +(-.25,0) .. (-1.25,.25);
\pgfsetlinewidth{.05*\pgflinewidth}

\draw (-1.45,-.51) .. controls +(.005,.05) and +(-.25,0) .. (-1.25,.25);

\begin{scope}[shift ={+(2.5,0)}]
\draw[fill, white] (-.75,-.5) rectangle (-1.75,.5);
\draw (-.935,-.51) .. controls +(-.005,.05) and +(.25,0) .. (-1.25,.25);
\pgfsetlinewidth{20*\pgflinewidth}
\draw[white]   (-.935,.51) .. controls +(-.005,.05) and +(.25,0) .. (-1.25,-.25);
\pgfsetlinewidth{.05*\pgflinewidth}
\draw (-.865,.71) .. controls +(-.005,.05) and +(.25,0) .. (-1.25,-.25);

\draw (-1.45,.51) .. controls +(-.005,.05) and +(-.25,0) .. (-1.25,-.25);
\pgfsetlinewidth{20*\pgflinewidth}
\draw[white]   (-1.45,-.51) .. controls +(.005,.05) and +(-.25,0) .. (-1.25,.25);
\pgfsetlinewidth{.05*\pgflinewidth}

\draw (-1.45,-.51) .. controls +(.005,.05) and +(-.25,0) .. (-1.25,.25);
\end{scope}

\begin{scope}[shift = {+(8,0)}]
\draw (-1.5,0) arc (180:0:2.5);
\draw (-1,0) arc (180:0:2);
\pgfsetlinewidth{20*\pgflinewidth}
\draw[white] (1.5,0) arc (180:90:2);
\pgfsetlinewidth{.05*\pgflinewidth}

\draw (1.5,0) arc (180:0:2);
\draw (1.5,0) arc (-180:0:2);

\pgfsetlinewidth{20*\pgflinewidth}
\draw[white]  (-1.5,0) arc (-180:0:2.5);
\draw[white] (-1,0) arc (-180:0:2);
\pgfsetlinewidth{.05*\pgflinewidth}

\draw (-1.5,0) arc (-180:0:2.5);
\draw (-1,0) arc (-180:0:2);

\draw[fill, white] (-.75,-.5) rectangle (-1.75,.5);

\draw (-.935,-.51) .. controls +(-.005,.05) and +(.25,0) .. (-1.25,.25);
\pgfsetlinewidth{20*\pgflinewidth}
\draw[white]   (-.935,.51) .. controls +(-.005,.05) and +(.25,0) .. (-1.25,-.25);
\pgfsetlinewidth{.05*\pgflinewidth}
\draw (-.865,.71) .. controls +(-.005,.05) and +(.25,0) .. (-1.25,-.25);

\draw (-1.45,.51) .. controls +(-.005,.05) and +(-.25,0) .. (-1.25,-.25);
\pgfsetlinewidth{20*\pgflinewidth}
\draw[white]   (-1.45,-.51) .. controls +(.005,.05) and +(-.25,0) .. (-1.25,.25);
\pgfsetlinewidth{.05*\pgflinewidth}

\draw (-1.45,-.51) .. controls +(.005,.05) and +(-.25,0) .. (-1.25,.25);

\end{scope}
\end{tikzpicture}
\caption{Diagrams of the (untwisted) fully clasped (bottom left) and (untwisted) reduced (bottom right) Whitehead doubles of the Hopf link (top left).}
\label{figure:example_of_whitehead_d}
\end{figure}
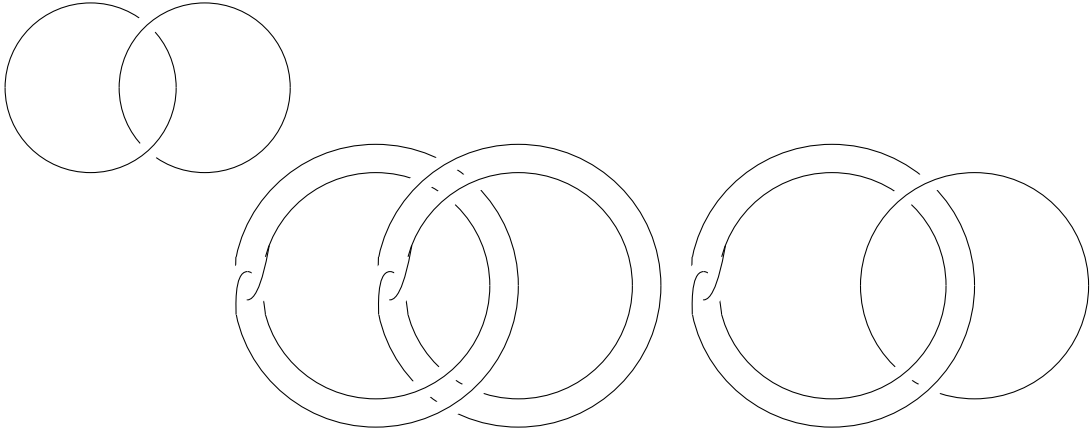

\subsection{Slice-torus link invariants of Whitehead doubles}

Now, we shall study how the slice-torus link invariants behave in the case of Whitehead doubles. Before proceeding, we observe that given a link $L$ and a slice-torus link invariant $\nu$, there are two functions
\[F_{\nu}(L)\: :\: \mathbb{Z}^\ell \longrightarrow \mathbb{R}\:\:\:\:\:\text{ and }\:\:\:\:\:F^\prime_{\nu}(L;L_{1})\: :\: \mathbb{Z} \longrightarrow \mathbb{R}\]
defined as
\[ F_{\nu}(L)(\underline{t})= \nu(W_{+}(L,\underline{t}))\quad\text{and}\quad F^\prime_{\nu}(L;L_{1})(t)= \nu(W^\prime_{+}(L,t;L_{1})).\]
Similarly, we can define $\overline{F}_{\nu}(L)$ and $\overline{F^\prime}_{\nu}(L)$ by using the negative Whitehead doubles. Since two equivalent links have equivalent Whitehead doubles, it follows immediately that these functions are link invariants
(where we identify the variables corresponding to isotopic components),
but we can say more. In fact, we have that all of these functions are also invariant under strong concordance.
\begin{teo}
 \label{teo:concordant}
 Let $L_1$ and $L_2$ be two $\ell$-component links which are strongly concordant. Consider $\underline t\in\Z^\ell$ and $t\in\Z$. Denote by $L'_1$ and $L'_2$ two components of $L_1$ and $L_2$ respectively.
 Then, $W_\pm(L_1,\underline t)$ and $W'_\pm(L_1,t;L'_1)$ are strongly concordant (respecting the ordering of the components) to $W_\pm(L_2,\underline t)$ and $W'_\pm(L_2,t;L'_2)$, respectively. In particular, the functions $F_{\nu},\overline F_{\nu},F'_{\nu}$ and $\overline F'_{\nu}$ are
 strong concordance invariants of links.
\end{teo}
\begin{proof}
 Suppose that the strong concordance between $L_1$ and $L_2$ appears like in Figure \ref{figure:cobordism}. Consider a movie (i.e. a sequence of band moves, birth and death of unknotted components, and Reidemeister moves) from a diagram of $L_1$ into one of $L_2$, describing a concordance.
We start by taking the fully clasped Whitehead doubles of $L_1$, obtained by doubling the given diagram of $L_1$ as we described before in this section. 
 Every birth move now becomes a double birth move, see Figure \ref{figure:two_zero_handles}, which corresponds to the attachment of two $0$-handles.
  \begin{figure}[h]
     \centering
\begin{tikzpicture}[scale = .6]
\node at (0,5) {$\emptyset$};
\draw[->] (1,5) -- (3,5);
\draw[->] (4,5) arc (180:-180:2);
\draw[<-] (5,5) arc (180:-180:1);
\end{tikzpicture}
\caption{A double birth move.}
\label{figure:two_zero_handles}
 \end{figure}
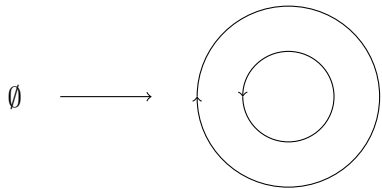
 Moreover, when we have a split move, the component involved will be doubled
 and then the move now consists of two band moves, instead of one, as shown in Figure \ref{figure:double_band}. Therefore, we have two cases, depending on whether the doubled component is clasped or not. If it is clasped then, after the band moves, it will be split into three components, one containing the clasp and the other two being one the double of the other. On the other hand, if the doubled component is not clasped then the bands will turn it into two doubled components.
\begin{figure}[h]
     \centering
\begin{tikzpicture}[scale = .5]
\begin{scope}[shift ={+(-20,0)}]
\draw[<-] (1,5) .. controls +(1,-1.5) and +(1,1.5) .. (1,0);
\draw[->] (5,1) .. controls +(-.75,1) and +(-.75,-1) .. (5,4);
\draw[<-] (0,1) .. controls +(.75,1) and +(.75,-1) .. (0,4);
\draw[<-] (4,0) .. controls +(-1,1.5) and +(-.75,-1.5) .. (4,5);

\draw[dashed] (1.75,2.5) -- (3.25,2.5);
\end{scope}

\draw[->] (-14,2.5) -- (-11,2.5);
\node[above] at ( -12.5,2.5) {Band move};

\begin{scope}[shift ={+(-10,0)}]
\draw[<-] (1,5) .. controls +(.75,-1) and +(-.75,-1) .. (4,5);
\draw[->] (5,1) .. controls +(-.75,1) and +(-.75,-1) .. (5,4);
\draw[->] (1,0) .. controls +(.75,1) and +(-.75,1) .. (4,0);
\draw[<-] (0,1) .. controls +(.75,1) and +(.75,-1) .. (0,4);
\draw[dashed] (.55,2.5) -- (4.55,2.5);
\end{scope}

\draw[->] (-4,2.5) -- (-1,2.5);
\node[above] at ( -2.5,2.5) {Band move};

\draw[<-] (1,5) .. controls +(.75,-1) and +(-.75,-1) .. (4,5);
\draw[->] (0,4) .. controls +(1,-1.5) and +(-1,-1.5) .. (5,4);
\draw[->] (1,0) .. controls +(.75,1) and +(-.75,1) .. (4,0);
\draw[<-] (0,1) .. controls +(1,1.5) and +(-1,1.5) .. (5,1);
\end{tikzpicture}
\caption{A double band move.}
\label{figure:double_band}
 \end{figure}
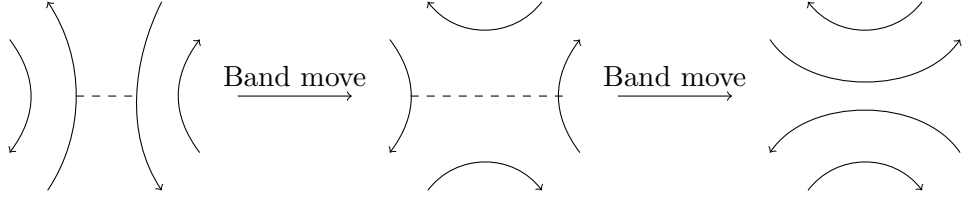

 Now in the case of a merge move, we observe that two clasped components cannot be joined together. In fact, otherwise our cobordism would not be a strong concordance. This implies that each merge move corresponds precisely to the inverse of a split move and then we obtain the same conclusions of the previous case.
 
 At this point we can perform the death moves, which will be doubled in the same way of the birth moves before. This is because the clasps, and the full twists, can be isotoped to be everywhere on their doubled component of $L_1$; therefore we can always be sure that clasps and twists will not appear on the components that we want to cancel with the death moves.
 
 After this procedure, we are left with a diagram of the fully clasped Whitehead double of $L_2$, with the same number of twists. Moreover, the new cobordism that we obtained is a strong concordance by construction.
 
 For the reduced Whitehead doubles the reasoning is exactly the same, provided that we take care of two more details.
 
 First, we only double the birth moves that will be joined with the clasped component and not the others.
 
 Second, we have to observe that we cannot have a merge move between a component that is doubled and one that is not. In fact, we start from a strong concordance and, as we remarked before, we cannot merge different components of $L_1$ together. 
 Then the claim follows from the same argument we used for fully clasped doubles.
\end{proof}

\begin{remark}
In the case $L$ is a knot, the functions $F_{\nu}(L)$ and $F^\prime_{\nu}(L;L)$ coincide. Furthermore, when $\nu$ is $\Z$-valued these functions assume only two values, and the point where their value changes is the $t_{\nu}$ invariant introduced in \cite{LivingstonNaik}.
\end{remark}

Let us start by proving that the functions we introduced are bounded.

\begin{teo}\label{teorema:bounds_on_t_nu}
For each $\ell$-component link $L$ and $\underline{m}\in \mathbb{Z}^\ell$, then
\[ F_{\nu}(L)(\underline{m}) \in [0, \ell]\quad \text{and}\quad  \overline{F}_{\nu}(L)(\underline{m}) \in [-\ell,0]. \]
Furthermore, given a component $L_{0}$ of $L$ and $m\in\Z$ we have
\[ F^\prime_{\nu}(L;L_{0})(m) \in [\nu_0, \nu_{0} + 1]\quad \text{and}\quad  \overline{F}^\prime_{\nu}(L;L_{0})(m) \in [\nu_0 - 1,\nu_0], \]
where
\[\nu_0 = \begin{cases} \nu(L\setminus L_{0}) & \ell \geq 2\\ 0 & \ell =1\end{cases}\]
\end{teo}
\begin{proof}
Let us prove only the part of statement concerning the invariants $F^\prime_{\nu}$ and $\overline{F}^\prime_{\nu}$. The rest of the statement can be proved by iterating the same reasoning.

Notice that we can obtain a diagram of  $(L\setminus L_{0})\sqcup H_{+}$ via a single band move on $D_{+}(L,m;L_{0})$ (similar to the one illustrated in Figure \ref{figure:band_move_W_double}). Thus, from Properties (A) and (B), and Lemma \ref{lemma:values_of_nu_1} it follows that
\[\nu(W_{+}(L,m;L_{0})) \leq \nu((L\setminus L_{0})\sqcup H_{+}) =  \nu_{0}+ 1 \]
and
\[ \nu_{0} =\nu_{0}+ 1 -1 = \nu((L\setminus L_{0})\sqcup H_{+})  - 1 \leq \nu(W_{+}(L,m;L_{0})). \]
The same reasoning applies for $\overline{F}^\prime_{\nu}$, the only change is that we get $H_{-}$ instead of $H_{+}$  (see Figure \ref{figure:band_move_W_double}). Since $\nu(H_{-}) = 0$, the result follows.
\begin{figure}[H]
\centering
\begin{tikzpicture}[scale =.5]
\draw[-> ] (2,0) arc (0:90:2)--(-4,2);
\pgfsetlinewidth{20*\pgflinewidth}
\draw[white] (0,0) arc (180:90:2);
\pgfsetlinewidth{.05*\pgflinewidth}
\draw[<- ] (0,0) arc (180:90:2) -- (4,2);

\draw[-> ] (0,0) arc (180:270:2)-- (4,-2);
\pgfsetlinewidth{20*\pgflinewidth}
\draw[white] (2,0) arc (0:-90:2);
\pgfsetlinewidth{.05*\pgflinewidth}
\draw[<- ] (2,0) arc (0:-90:2) -- (-4,-2);

\draw[dashed] (-2,-2) -- (-2,2);

\draw[->] (6,0) -- (8,0);
\node at (7,.5) {Band};
\node at (7,-.5) {move};

\begin{scope}[shift ={+(14,0)}]
\draw (2,0) arc (0:180:2);
\pgfsetlinewidth{20*\pgflinewidth}
\draw[white] (0,0) arc (180:90:2);
\pgfsetlinewidth{.05*\pgflinewidth}
\draw[<- ] (0,0) arc (180:90:2) -- (4,2);

\draw[-> ] (0,0) arc (180:270:2)-- (4,-2);
\pgfsetlinewidth{20*\pgflinewidth}
\draw[white] (2,0) arc (0:-90:2);
\pgfsetlinewidth{.05*\pgflinewidth}
\draw[<- ]  (2,0) arc (0:-180:2) ;
\draw[<-] (-5,2,0) arc (90:-90:2);
\end{scope}
\end{tikzpicture}
\caption{A band move between $W_{-}(L,m;L_{0})$ and $(L\setminus L_{0})\sqcup H_{-}$.}\label{figure:band_move_W_double}
\end{figure}
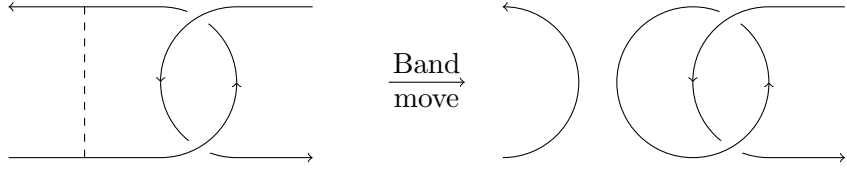
\end{proof}

Moreover, there is a non-increasing property akin to the one proved in \cite{LivingstonNaik}. 

\begin{teo}
\label{teo:decreasing}
Let $L$ be an oriented link, and let $\nu$ be a slice-torus link invariant. If $\underline{m}$ and $\underline{n}$ are two elements of $\mathbb{Z}^\ell$ such that $m_{i} \geq n_{i}$ for all $i$, then 
\[ F_{\nu}(L)(\underline{n}) - \sum_{i=1}^\ell (m_{i} - n_{i}) \leq F_{\nu}(L)(\underline{m}) \leq F_{\nu}(L)(\underline{n})\:.\]
Furthermore, if $m\geq n$ are two integers then
\[ F^\prime_{\nu}(L;L_0)(n)  -  (m - n) \leq F^\prime_{\nu}(L;L_{0})(m) \leq F^\prime_{\nu}(L;L_0)(n) \]
Moreover, the same result holds for $\overline{F}_{\nu}$ and $\overline{F}^\prime_{\nu}$.
\end{teo}
\begin{proof}
It is sufficient to prove the result for the case $m_{i} = n_{i}$, for all $i \ne i_{0}$, and $m_{i_0} = n_{i_0} + 1$. It is sufficient to notice that, in this case, one may obtain $W_{\pm}(L,\underline{m})$ (resp. $W^\prime_{\pm}(L,m;L_0)$) from  $W_{\pm}(L,\underline{n})$ (resp. $W^\prime_{\pm}(L,n;L_0)$) by a second Reidemeister move and a crossing change from a positive crossing to a negative crossing (see Figure \ref{figure:crossing_change_increasing}) and the result follows from Proposition \ref{proposition:crossing_change}.
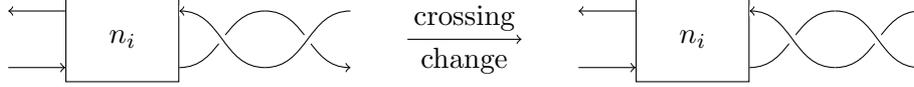
\begin{figure}[H]
\centering
\begin{tikzpicture}[scale = .75]
\begin{scope}[shift ={+(-10,0)}]

\draw[<-] (8,.75) -- (7,.75);
\draw[->] (8,1.75) -- (7,1.75);
\draw (9.99,2) rectangle (8.01,.5);
\node at (9,1.25) {$n_{i}$};

\draw[->] (11.5,1.75)  .. controls +(.75,0) and +(-.75,0) .. (13,.75) ;
\draw (10,.75)  .. controls +(.75,0) and +(-.75,0) .. (11.5,1.75) ;
\pgfsetlinewidth{10*\pgflinewidth}
\draw[white]  (10,1.75)  .. controls +(.75,0) and +(-.75,0) .. (11.5,.75);
\draw[white] (11.5,.75)  .. controls +(.75,0) and +(-.75,0) .. (13,1.75) ;
\pgfsetlinewidth{.1*\pgflinewidth}
\draw[<-] (10,1.75)  .. controls +(.75,0) and +(-.75,0) .. (11.5,.75);
\draw (11.5,.75)  .. controls +(.75,0) and +(-.75,0) .. (13,1.75) ;
\end{scope}
\draw[<-] (8,.75) -- (7,.75);
\draw[->] (8,1.75) -- (7,1.75);
\draw (9.99,2) rectangle (8.01,.5);
\node at (9,1.25) {$n_{i}$};

\draw (11.5,.75)  .. controls +(.75,0) and +(-.75,0) .. (13,1.75) ;
\draw (10,.75)  .. controls +(.75,0) and +(-.75,0) .. (11.5,1.75) ;
\pgfsetlinewidth{10*\pgflinewidth}
\draw[white]  (10,1.75)  .. controls +(.75,0) and +(-.75,0) .. (11.5,.75);
\draw[white]  (11.5,1.75)  .. controls +(.75,0) and +(-.75,0) .. (13,.75) ;
\pgfsetlinewidth{.1*\pgflinewidth}
\draw[<-] (10,1.75)  .. controls +(.75,0) and +(-.75,0) .. (11.5,.75);
\draw[->] (11.5,1.75)  .. controls +(.75,0) and +(-.75,0) .. (13,.75) ;

\draw[<-] (6,1.25) -- (4,1.25);
\node at (5,1.625) {crossing};
\node at (5,.875) {change};

\end{tikzpicture}
\caption{How obtain $W_{\pm}(L,\underline{m})$ (resp. $W^\prime_{\pm}(L,m;L_0)$) from  $W_{\pm}(L,\underline{n})$ (resp. $W^\prime_{\pm}(L,n;L_0)$) by a second Reidemeister move and a crossing change.}\label{figure:crossing_change_increasing}
\end{figure}
\end{proof}

In particular, Theorem \ref{teo:decreasing} implies that every time we add a positive full twist on one component the value of $\nu$ cannot increase, and it may decrease at most by one.

From the fact that $\nu(K) = - \nu(-K^*)$, for each knot $K$ and slice-torus invariant $\nu$, it follows that $\overline F_{\nu}(K)(t) = -F_{\nu}(K^*)(-t)$. In the case of links this property does not hold, but the invariants $F_{\nu}$ and $\overline{F}_{\nu}$ (and their ``reduced'' variants) still share the following similar, albeit weaker, symmetry property which is made precise in the following result.

\begin{prop}
\label{prop:symmetry}
Let $L$ be an $\ell$-component oriented link and let $\nu$ be a slice-torus link invariant. If there exists $\underline{m}$ (resp. $\underline{n}$) in $\mathbb{Z}^\ell$ such that
\[F_{\nu}(L)(\underline{m}) = \ell\qquad\text{( resp. } \overline{F}_{\nu}(L)(\underline{n})= -\ell\text{)} ,\]
then
\[\overline{F}_{\nu}(L)(\underline{m}) = 0 \qquad\text{( resp. } \quad F_{\nu}(L)(\underline{n})= 0 \text{)}\]
Moreover, let $L_0$ be a component of $L$, if there exists $m$ (resp. $n$) in $\mathbb{Z}$ such that
\[F^\prime_{\nu}(L;L_0)(m) = \nu_{0} +1\qquad\text{( resp. }  \overline{F}^\prime_{\nu}(L;L_0)(n)= \nu_0-1 \text{)},\]
then
\[\overline{F}^\prime_{\nu}(L;L_0)(m) = \nu_0\qquad\text{( resp. } F^\prime_{\nu}(L;L_0)(n)= \nu_0\text{)}, \]
where $\nu_{0}$ is defined as in Theorem \ref{teorema:bounds_on_t_nu}.
\end{prop}
\begin{proof}
We shall prove only the parts of the statement concerning the functions $F^\prime_{\nu}$ and $\overline{F}^\prime_{\nu}$. The proofs for the other functions are completely analogous and thence left to the reader.
Let us start by considering the sequence of band moves (and isotopies) depicted in Figure \ref{figura:bands_doubles}. 
\begin{figure}[H]
\centering
\begin{tikzpicture}[scale = .5]
\begin{scope}[shift ={+(-18,0)}]

\draw(3,2.5) .. controls +(0,-.5) and +(1.5,1.5) .. (1.5,0);
\pgfsetlinewidth{20*\pgflinewidth}
\draw[white] (2,2.5) .. controls +(0,-.5) and +(-1.5,1.5) .. (3.5,0);
\pgfsetlinewidth{.05*\pgflinewidth}
\draw[->] (2,2.5) .. controls +(0,-.5) and +(-1.5,1.5) .. (3.5,0);

\draw (3.5,5) .. controls +(-1.5,-1.5) and +(0,.5) .. (2,2.5);
\pgfsetlinewidth{20*\pgflinewidth}
\draw[white] (1.5,5) .. controls +(1.5,-1.5) and +(0,.5) .. (3,2.5);
\pgfsetlinewidth{.05*\pgflinewidth}
\draw[<-] (1.5,5) .. controls +(1.5,-1.5) and +(0,.5) .. (3,2.5);

\draw[dashed] (2,2.5) -- (3,2.5);
\end{scope}

\draw[->] (-14,2.5) -- (-11,2.5);
\node[above] at ( -12.5,2.5) {Band move};

\begin{scope}[shift ={+(-12,0)}]

\draw(2.5,2) .. controls +(.75,0) and +(1.5,1.5) .. (1.5,0);
\pgfsetlinewidth{20*\pgflinewidth}
\draw[white] (2.5,2) .. controls +(-.75,0) and +(-1.5,1.5) .. (3.5,0);
\pgfsetlinewidth{.05*\pgflinewidth}
\draw[->] (2.5,2) .. controls +(-.75,0) and +(-1.5,1.5) .. (3.5,0);

\draw (3.5,5) .. controls +(-1.5,-1.5) and +(-.75,0) .. (2.5,3);
\pgfsetlinewidth{20*\pgflinewidth}
\draw[white] (1.5,5) .. controls +(1.5,-1.5) and +(.75,0) .. (2.5,3);
\pgfsetlinewidth{.05*\pgflinewidth}
\draw[<-] (1.5,5) .. controls +(1.5,-1.5) and +(.75,0) .. (2.5,3);

\end{scope}

\draw[->] (-8,2.5) -- (-5,2.5);
\node[above] at ( -6.5,2.5) {$R_1$ moves};

\begin{scope}[shift ={+(-6,0)}]

\draw[->] (2.5,2) .. controls +(-.75,0) and +(-1.5,1.5) .. (3.5,0);
\pgfsetlinewidth{20*\pgflinewidth}
\draw[white] (2.5,2) .. controls +(.75,0) and +(1.5,1.5) .. (1.5,0);
\pgfsetlinewidth{.05*\pgflinewidth}
\draw(2.5,2) .. controls +(.75,0) and +(1.5,1.5) .. (1.5,0);

\draw[<-] (1.5,5) .. controls +(1.5,-1.5) and +(.75,0) .. (2.5,3);
\pgfsetlinewidth{20*\pgflinewidth}
\draw[white] (3.5,5) .. controls +(-1.5,-1.5) and +(-.75,0) .. (2.5,3);
\pgfsetlinewidth{.05*\pgflinewidth}
\draw (3.5,5) .. controls +(-1.5,-1.5) and +(-.75,0) .. (2.5,3);

\draw[dashed] (2.5,2) -- (2.5,3);

\end{scope}

\draw[->] (-2,2.5) -- (1,2.5);
\node[above] at ( -.5,2.5) {Band move};

\draw[->] (2,2.5) .. controls +(0,-.5) and +(-1.5,1.5) .. (3.5,0);
\pgfsetlinewidth{20*\pgflinewidth}
\draw[white] (3,2.5) .. controls +(0,-.5) and +(1.5,1.5) .. (1.5,0);
\pgfsetlinewidth{.05*\pgflinewidth}
\draw(3,2.5) .. controls +(0,-.5) and +(1.5,1.5) .. (1.5,0);

\draw[<-] (1.5,5) .. controls +(1.5,-1.5) and +(0,.5) .. (3,2.5);
\pgfsetlinewidth{20*\pgflinewidth}
\draw[white] (3.5,5) .. controls +(-1.5,-1.5) and +(0,.5) .. (2,2.5);
\pgfsetlinewidth{.05*\pgflinewidth}
\draw (3.5,5) .. controls +(-1.5,-1.5) and +(0,.5) .. (2,2.5);
\end{tikzpicture}
\caption{A sequence of band moves and isotopies relating  positive and negative clasps.}\label{figura:bands_doubles}
\end{figure}
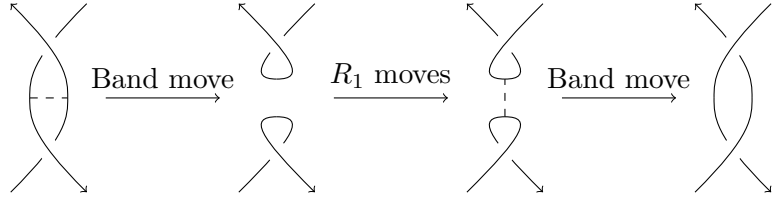
By Property (A) we obtain that
\[\vert\: \nu(W_{+}^\prime(L,m;L_0)) - \nu(W_{-}^\prime(L,m;L_0))\: \vert \leq 1.\]
It follows that, if $ F^\prime_{\nu}(L;L_0)(n)= \nu_0 + 1$ (resp. $\overline{F}^\prime_{\nu}(L;L_0)(m) = \nu_0 - 1$), then 
\[ \nu_{0} \leq\overline{F}^\prime_{\nu}(L;L_0)(m) \quad \text{(resp. }F^\prime_{\nu}(L;L_0)(n)\leq \nu_0\text{).} \]
Theorem \ref{teorema:bounds_on_t_nu} provides the other half of the bound(s), and the equality follows.
\end{proof}
We saw that $\nu(W_\pm(L,\underline t))$ is bounded for every $\ell$-component link $L$. Now we want to prove that for some $\ell$-tuples $(t_1,...,t_{\ell})$ the invariant 
$\nu$ assumes the maximum value possible.
To do this we equip $\mathbb{S}^3$ with the standard contact structure $\xi_{\text{st}}$ and we recall that the Thurston-Bennequin number 
$\text{tb}(\mathcal L)$ of a Legendrian knot
is the linking number between $L$ and the contact framing $L_{\xi_{\text{st}}}$ induced by $\xi_{\text{st}}$. See \cite{Etnyre} for details.
\begin{teo}
 \label{teo:maximum}
 Suppose that $\nu$ is a slice-torus link invariant. Then, for each Legendrian representative $\mathcal L$ of 
 $L$ in $(\mathbb{S}^3,\xi_{\emph{st}})$, and $(t_1,...,t_{\ell})$ such that 
 $t_i=\emph{tb}(\mathcal L_i)$ for every $i=1,...,\ell$, we have that \[ F_{\nu}(L)(t_1,...,t_{\ell})=\ell\:.\]
\end{teo}
\begin{proof}
 Let us consider a surface $F'$ such that the Legendrian link $\mathcal L$ in $\mathbb S^3$, equipped with the contact structure $\xi_{\text{st}}$, is embedded in $F'$ and $\text{tb}(\mathcal L)$ coincides with the Seifert framing induced by $F'$. 
 This means that the $(t_1,...,t_{\ell})$-twisted double of the link $L$ can be embedded in $F'$ as the boundary of a collar neighbourhood of $\mathcal L$. Let us call $F''\subset F'$ this neighbourhood. Moreover, we can change $F''$ by positive Hopf plumbings in such a way that the new surface $F$ has $W_+(L,t_1,...,t_\ell)$ as boundary.
 
 From the work of Rudolph \cite{Rudolph1,Rudolph2}, and thanks to our choice of $(t_1,...,t_{\ell})$, we can assume that $F''$ is a quasi-positive surface. Therefore, since $F$ it is obtained from $F''$ through positive Hopf plumbings, also $F$ is a quasi-positive surface. This also implies that $W_+(L,t_1,...,t_\ell)$ is a strongly quasi-positive link.
 
Now, possibly after performing more positive plumbings, the surface $F$ can be seen as a subsurface of a minimal Seifert surface $G$ of a torus knot $T_{m,n}$, for $m,n$ sufficiently big. (See \cite{Rudolph1,Rudolph2}.)
 The surface $\Sigma=G\setminus\mathring F$ is a connected cobordism between $W_+(L,t_1,...,t_\ell)$ and $T_{m,n}$ (this is basically the same argument used in \cite{LivingstonNaik}); its genus can be computed from the fact that
 $\chi(F)+\chi(\Sigma)=\chi(G)$. Hence, we have 
 \begin{equation}
  \label{torus}
  (-\ell)+(1-2g(\Sigma)-\ell)=1-2g(G)\:,
 \end{equation} 
because $F$ is by construction the union of $\ell$ tori with a disk removed, and $\Sigma$ and $G$ are both connected, which gives
$$g(\Sigma)=g(G)-\ell\:.$$
Proposition \ref{proposition:slice_genus_bound} and Proposition \ref{corollary:bound_on_the_slice_genus} imply that
\[ \nu(T_{m,n})\leq\nu(W_+(L,t_1,...,t_\ell))+g(\Sigma)=\nu(W_+(L,t_1,...,t_\ell))+g(G)-\ell\:,\]
which, in turn, gives 
\[ \ell\leq\nu(W_+(L,t_1,...,t_\ell))\:,\]
since $\nu(T_{m,n})=g(G)$ (cf. Corollary \ref{corollary:slice_genus_pos_torus_links}).
Now, the statement follows directly from Theorem \ref{teorema:bounds_on_t_nu}. 
\end{proof}
The last theorem, together with Proposition \ref{prop:symmetry}, immediately implies the following corollary.
\begin{cor}
 \label{cor:maximum}
 Let $L$ be a link. For every $\underline t\in\Z^\ell$ as in the hypotheses of Theorem \ref{teo:maximum}, we have that $\overline{F}_\nu(L)(\underline t)=0$.
\end{cor}
In \cite{LivingstonNaik} it was proved that, in the case of knots, the functions $F_{\nu}$ and $\overline F_{\nu}$ also assume the minimal value. That is, we have the following proposition.
\begin{prop}[Livingston and Naik,  \cite{LivingstonNaik}]
 If $K$ is a knot then there exists an integer $t$ such that $$\nu(W_+(K,t))=0\:\:\:\:\text{and}\:\:\:\:\nu(W_-(K,t))=-1\:.$$
\end{prop}
In general, we do not have the same result in the case of multi-component links. That is to say, we cannot prove that any of the functions we introduced reach the minimum. However, we can prove that some of them are non-constant.
\begin{prop}
 \label{prop:step}
 For every slice-torus link invariant $\nu$ and $\ell$-component link $L$ there exists an $\ell$-tuple $(t_1,...,t_{\ell})$ such that 
 \[F_{\nu}(L)(t_1,...,t_i,...,t_{\ell})\in(\ell-1,\ell\:]\:\:\:\:\:\text{and}\:\:\:\:\:F_{\nu}(L)(t_1,...,t_i+1,...,t_{\ell})\in(\ell-2,\ell-1\:]\]
 for some $i$, and the same holds true for $\overline F_{\nu}(L)$.
\end{prop}
\begin{proof}
Suppose that $F_{\nu}(L)$ has values only in the interval $(\ell-1,\ell\:]$. Then, by Property (C) we have that \[\nu(W_+(L,\underline t)^*)=\nu(W_-(L^*,-\underline t))\leq-1\] for each $\underline t\in\Z^\ell$, but this contradicts Corollary \ref{cor:maximum}. The claim follows from Theorem \ref{teo:decreasing}. The case of $\overline F_{\nu}(L)$ is dealt with in the same way.
\end{proof}

\section{An example: the Hopf link}\label{section:example}

In this section we will give compute explicitly the functions $F_{\nu_{s}}$ and $\overline{F}_{\nu_{s}}$ for the Hopf link, where $\nu_{s}$ is the slice-torus link invariant associated to $s$. Notice that the fully clasped Whitehead doubles of the positive and negative Hopf links are isotopic, and thus the computation we achieve are valid both for $H_{+}$ and $H_{-}$. To lighten the notation, throughout the section we shall omit $H_{\pm}$ from the notation unless confusion may arise.

\subsection{Computations for general slice-torus link invariant}
Let $\nu$ be a slice-torus link invariant.
There are a few observations on $F_{\nu}(t_{1},t_{2})$ and $\overline{F}_{\nu}(t_{1},t_{2})$ which can be made, and allow us to partially compute these functions.

First, notice that exchanging the roles of the components of the Hopf link yields the same link. 
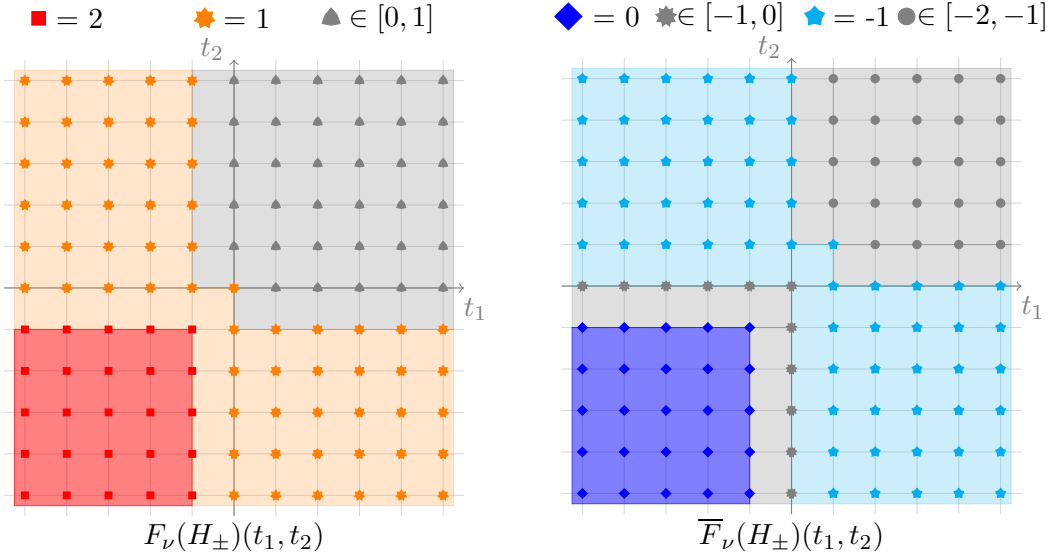
\begin{figure}[H]
    \centering

\begin{tikzpicture}[scale=.55]
\draw[opacity=.3, gray] (-5.5,-5.5) grid (5.5,5.5);
\draw[gray,  ->] (-5.5,0) -- (5.5,0);
\draw[gray, ->] (0,-5.5) -- (0,5.5);

\draw[fill, opacity=.2, orange] (-5.25,5.25) --(-5.25,-1) -- (-1,-1)--(-1,-5.25) --(5.25,-5.25) --(5.25,-1) --(0,-1)-- (0,0) -- (-1,0) --  (-1,5.25) -- cycle;
\draw[gray, fill, opacity = .25] (5.25,0) rectangle (0,-1);
\draw[gray, fill, opacity = .25] (0,5.25) rectangle (-1,0);
\draw[fill, opacity=.5, red] (-1,-5.25) rectangle (-5.25,-1);

\draw[fill, opacity=.25,gray] (0,0) rectangle (5.25,5.25);
\node[gray, above left] at (0,5.25) {$t_{2}$};
\node[gray, below right] at (5.25,0) {$t_{1}$};
\node[] at (0,-6) {$F_{\nu}(H_{\pm})(t_{1},t_{2})$};
\foreach \x in {-5,...,-1}
  {
  \node[draw,star,star points=3, gray,fill,scale=.25] at (-\x,0) {};
  \node[draw,star,star points=3, gray,fill,scale=.25] at (0,-\x) {};
  
  \foreach \y in {-5,...,-1}
  {
  \node[draw,rectangle, red,fill,scale=.35] at (\x,\y) {};
  \node[draw,star,star points=3, gray,fill,scale=.25] at (-\x,-\y) {};
  }
    \foreach \y in {0,...,5}
  {
  \node[draw,star,star points=7, orange,fill,scale=.25] at (\x,\y) {};
  \node[draw,star,star points=7, orange,fill,scale=.25] at (\y,\x) {};
  }
  }
  
\node[draw,star,star points=7, orange,fill,scale=.25] at (0,0) {};

\node[draw,star,star points=3,gray,fill,scale=.5, left] at (2.5,6.5) {};
\node[right] at (2.5,6.5) {$\in [0,1]$};
\node[draw,star,star points=7, orange,fill,scale=.5, left] at (-.5,6.5) {};
\node[right] at (-.5,6.5) {= 1};
\node[draw,rectangle, red,fill,scale=.7, left] at (-4.5,6.5) {};
\node[right] at (-4.5,6.5) {= 2};
\end{tikzpicture}
\hspace{.5cm}
\begin{tikzpicture}[scale=.55]
\draw[opacity=.3, gray] (-5.5,-5.5) grid (5.5,5.5);
\draw[gray,  ->] (-5.5,0) -- (5.5,0);
\draw[gray, ->] (0,-5.5) -- (0,5.5);

\draw[gray, fill, opacity = .25] (-5.25,0) rectangle (0,-1);
\draw[gray, fill, opacity = .25] (0,-5.25) rectangle (-1,-1);
\draw[gray, fill, opacity = .25] (5.25,0) rectangle (1,1);
\draw[gray, fill, opacity = .25] (0,5.25) rectangle (1,1);
\draw[fill, opacity=.2, cyan] (-5.25,5.25) --(-5.25,0) -- (0,0) --(0,-5.25) --(5.25,-5.25) --(5.25,0)--(1,0) -- (1,1) --(0,1) -- (0,5.25) -- cycle;

\draw[fill, opacity=.5, blue] (-1,-5.25) rectangle (-5.25,-1);

\draw[fill, opacity=.25,gray] (1,1) rectangle (5.25,5.25);
\node[gray, above left] at (0,5.25) {$t_{2}$};
\node[gray, below right] at (5.25,0) {$t_{1}$};
\node[] at (0,-6) {$\overline{F}_{\nu}(H_{\pm})(t_{1},t_{2})$};

\foreach \x in {-5,...,-1}
  {
  \node[draw,star,star points=9, gray,fill,scale=.25]  at (\x,0) {};
  \node[draw,star, star points=9, gray,fill,scale=.25]  at (0,\x) {};
  
  \foreach \y in {-5,...,-1}
  {
  \node[draw,diamond, blue,fill,scale=.25] at (\x,\y) {};
  \ifthenelse{\x < -1}{\draw[fill,gray] (-\x,-\y) circle (.1);}{\ifthenelse{\y < -1}{\draw[fill,gray] (-\x,-\y) circle (.1);}{}}
  }
    \foreach \y in {0,...,5}
  {
  \node[draw,star, cyan,fill,scale=.25] at (-\x,-\y ) {};
  \node[draw,star, cyan,fill,scale=.25] at (-\y,-\x) {};
  }
  }

\node[draw,star, star points=9, gray,fill,scale=.25] at (0,0) {};
\node[draw,star, cyan,fill,scale=.25] at (1,1) {};

\draw[fill,gray] (2.75,6.5) circle (.2);
\node[right] at (2.75,6.5) {$\in[-2,-1]$};
\node[draw,star, star points=9, gray,fill,scale=.5] at (-3,6.5) {};
\node[right] at (-3,6.5) {$\in [-1,0] $};
\node[draw,star, cyan,fill,scale=.5, left] at (.75,6.5) {};
\node[right] at (.75,6.5) {= -1};
\node[draw,diamond, blue,fill,scale=.7, left] at (-5,6.5) {};
\node[right] at (-5,6.5) {= 0};
\end{tikzpicture}
\caption{Partial computations of the functions $F_{\nu}(H_{\pm})(t_{1},t_{2})$ and $\overline{F}_{\nu}(H_{\pm})(t_{1},t_{2})$ for an arbitrary slice-torus link invariant $\nu$.}
\label{figure:general_F_for_H_+}
\end{figure}
Thus, we obtain the following symmetry property
\[ F_{\nu}(H_\pm)(t_1,t_2)=F_{\nu}(H_\pm)(t_2,t_1),\qquad \forall (t_1,t_2) \in \mathbb{Z}^{2}.\] 

Furthermore, with the same reasoning as in the proof of Theorem \ref{teorema:bounds_on_t_nu}, it follows immediately that
\[\nu (W_{+}(\bigcirc, t_{i}))\leq F_{\nu}(t_{1},t_{2}) \leq \nu (W_{+}(\bigcirc, t_{i})) + 1 ,\qquad i \in \{ 1,2 \}.\]
Moreover, since $W_{+}(\bigcirc ) (-1) = T_{2,3}$ and $W_{+}(\bigcirc ) (0) = \bigcirc$, we have that
\[\nu (W_{+}(\bigcirc, t))=F_{\nu}(\bigcirc)(t) =\begin{cases} 1 & t \leq -1 \\ 0 & t \geq 0 \end{cases}, \]
for each $\nu$. Putting these facts together with Theorem \ref{teorema:bounds_on_t_nu}, and with the fact that the combinatorial bound (cf. Theorem \ref{teorema:bound_combinatorio}) is sharp in the case $(t_1,t_2)= (0,0)$, we obtain that
\[F_{\nu}(t_{1},t_{2}) = \begin{cases} 2 & (t_1,t_2) \in (-\infty,-1] \times (-\infty,-1] \\ 1 & (t_1,t_2)\ \text{or}\ (t_2,t_1) \in [-1,-\infty) \times (-\infty,0] \cup \{ (0,0) \}\\ \in [0,1] & \text{otherwise.} \end{cases}.\]
A similar, reasoning works with $\overline{F}_{\nu}$, with the only difference that the bound is sharp in $(1,1)$ and not in $(0,0)$. This lead us to the following
\[\overline{F}_{\nu}(t_{1},t_{2}) = \begin{cases} 0 & (t_1,t_2) \in (-\infty,-1] \times (-\infty,-1] \\ \in[-1,0] & (t_1,t_2)\ \text{or}\ (t_2,t_1) \in \{ 0 \} \times (-\infty,0]\\-1 & (t_1,t_2)\ \text{or}\ (t_2,t_1) \in [1,+\infty) \times (-\infty,0] \cup \{ (1,1)\}\\ \in[-2,-1] & \text{otherwise.} \end{cases}.\]
The information we gathered on the functions $F_{\nu}$ and $\overline{F}_{\nu}$ is summarized in Figure \ref{figure:general_F_for_H_+}.
\begin{remark}\label{rem:conctosplit}
This amount of information is already enough to distinguish the unlink with two components, and every disjoint union of two knots, from the Hopf link by applying Theorem \ref{teo:split}.
\end{remark}

\subsection{Computations with the $s$-invariant}

The $s$-invariant was introduced by Rasmussen (\cite{Rasmussen}), in the case of knots, and extended to links by Beliakova and Werhli (\cite{BeliakovaWehrli}). Let us say a few words on this invariant.

Fix a field $\mathbb{F}$, in \cite{Lee05} E. S. Lee introduced a link homology theory $H_{Lee}^{*}(\cdot,\mathbb{F})$, which is a deformation of Khovanov homology. The homology of this theory is pretty simple: given an oriented link diagram $D$ representing a link $L$ there is a set of cycles, called \emph{canonical generators}, whose homology classes generate $H_{Lee}^{*}(L, \mathbb{F})$. This set is indexed by the possible orientations of the underlying unoriented diagram (\cite[Theorem 5.1]{Lee05}). Moreover, the homological degree $h$ of each canonical generator is completely determined by the linking matrix of $L$. However, this theory has a natural (decreasing) filtration $\mathcal{F}_{*}$, called the \emph{quantum filtration}, which contains non-trivial information on concordance.

Let $D$ be an oriented link diagram. The set of the possible orientations of the underlying unoriented diagram shall be denoted by $\mathbb{O}(D)$, and the canonical generator associated to a given $o\in \mathbb{O}(D)$ shall be denoted by $\mathbf{v}_{o}(D;\mathbb{F}) \in C_{Lee}^{h(o,L)}(D,\mathbb{F})$. 

\begin{defin}[Rasmussen \cite{Rasmussen}, Beliakova-Wehrli \cite{BeliakovaWehrli}]
Let $D$ be an oriented link diagram representing an oriented link $L$. The \emph{Rasmussen-Beliakova-Wehrli} (\emph{RBW}) \emph{invariant} associated to $o\in \mathbb{O}(D)$ is the integer
\[ s(o,L; \mathbb{F}) = \frac{Fdeg\left([\mathbf{v}_{o}(D;\mathbb{F}) - \mathbf{v}_{-o}(D;\mathbb{F})\right]) - Fdeg\left([\mathbf{v}_{o}(D;\mathbb{F}) + \mathbf{v}_{-o}(D;\mathbb{F})]\right)}{2},\]
where $Fdeg$ indicates the filtered degree in $H_{TLee}^{\bullet}(L, \mathbb{F})$, and $-o$ denotes the opposite orientation with respect to $o$. If $o$ is exactly the orientation induced by $L$, we will omit $o$ from the notation and call $s(L;\mathbb{F})$ the \emph{$s$-invariant} or \emph{Rasmussen invariant} of $L$.
\end{defin}

\begin{remark}
The original definition of the RBW-invariants (\cite{BeliakovaWehrli,Rasmussen}) does not work over fields of characteristic $2$. However, using a twisted version of Lee theory, defined by Bar-Natan in \cite{BarNatan05cob}, one can extend the definition to characteristic $2$ (of course, the two theories give the same invariants if $char(\mathbb{F}) \ne 2$, see \cite{Mackaayturnervaz07}). With an abuse of notation we shall call these extended invariants RBW-invariants.
\end{remark}

Let $D$ be an oriented link diagram representing the oriented link $L$. It can be easily shown (see, for instance, \cite[Proposition 11]{Collari} and subsequent proof) that
\[s(L;\mathbb{F}) - 1 =Fdeg([\mathbf{v}_{-o_{D}}(D)]) = Fdeg([\mathbf{v}_{o_{D}}(D)]) = \max \{ Fdeg(x)\: \vert\: x\in [\mathbf{v}_{o_{D}}(D)]\},\]
where $o_{D}$ is the orientation of $D$, and the filtered degree $Fdeg(x)$ is defined as the maximal $j$ such that $x\in \mathcal{F}_{j}C_{Lee}^{*}(D,\mathbb{F})$. We shall make use of this alternative definition of $s$ to prove the following results, which allows one to compute $s$ in a number of cases.

\begin{prop}
\label{teo:s}
Let $L$ be an oriented link, $D$ an oriented diagram representing $L$ and $\widetilde{L}$ the unoriented link underlying $L$. If for each $o\in \mathbb{O}(D) $ such that $\mathbf{v}_{o}(D)\in C_{Lee}^{0}(D,\mathbb{F})$ we have that $L$ is isotopic to $(\widetilde{L},o)$, then
\[ s(L;\mathbb{F}) =1+ \min\left\{ j\in \mathbb{Z}\: \big\vert\: Gr_{\mathcal{F}}^{j}H_{Lee}^{0}(L;\mathbb{F}) \ne 0 \right\},\]
where $Gr^*_\mathcal{F}$ indicates the associated graded object corresponding to the quantum filtration.
\end{prop}
\begin{proof}
The homology classes of the canonical generators associated to the orientations satisfying the above hypothesis, generate $H^{0}_{Lee}(L;\mathbb{F})$. Since $L$ is isotopic to $(\widetilde{L},o)$, for all $o$'s such that $\mathbf{v}_{o}(D)\in C_{Lee}^{0}(D,\mathbb{F})$, it follows that all the corresponding $[\mathbf{v}_{o}]$'s have the same filtered degree. The set of such $[\mathbf{v}_{o}]$'s  is a basis of $H_{TLee}^{0}(L;\mathbb{F})$, and the minimal filtered degree of the elements of a basis of a filtered vector space  does not depend on the choice of the basis (this fact is easy to prove, but the lazy reader can consult, for example, \cite[Corollary A.6]{Thesis1}). The claim follows immediately from the fact that the minimal degree of the elements of a filtered basis\footnote{A \emph{filtered basis} of a filtered vector space $V$ is a basis $\{ e_{i} \}_{i=1,...,k}$ for $V$ such that the direct sum filtration on $V = \bigoplus_{i=1}^{k} \mathbb{F}\left\langle e_{i}\right\rangle$ coincides with the original filtration (where the filtration on $\mathbb{F}\left\langle e_{i}\right\rangle$ is understood). It is clear that every filtered vector space admits a filtered basis, and this choice gives an isomorphism between $V$ and $Gr^{*}V$.} of a filtered vector space $V$ is the minimal degree where $Gr^*V$ is non trivial.
\end{proof}

\begin{remark}\label{remark:noPardon}
The proof of the previous proposition does not imply anything about the support of the associated graded object (called also the Pardon invariant), aside the RBW invariant being the lowest non-trivial quantum degree (plus one) in homological degree $0$. In particular, we did not prove that $Gr_{\mathcal{F}}^{*}H_{Lee}^{0}(L;\mathbb{F})$ is supported only in two degrees, which is false as we shall see in the examples.
\end{remark}
Using Proposition \ref{teo:s} and the knight move pairings (\cite{Thesis1,Lee05}) we can determine the $s$-invariant from Khovanov homology in many cases, including the Whitehead doubles of the Hopf link. An essential data to perform a computation using the knight move pairing is the homological degrees on which Lee homology is supported. These are completely determined by the linking matrix. In our case, since the linking matrix of fully clasped Whitehead doubles is always vanishing, Lee homology is always concentrated in homological degree $0$.

\begin{remark}
We remark that in general the knight move pairing is not sufficient to compute the associated graded object to Lee homology. However, if the field is of characteristic different from $2$, and the pairing is unique, this method can be used (see \cite[Chapter 2 \& Appendix B]{Thesis1} for more details).
\end{remark}

In order to avoid technical difficulties we shall work with $\mathbb{F} = \mathbb{Q}$. In this case we are able complete the computation started at the beginning of the section. 

We recall (see Example \ref{ex:first}) that the Rasmussen invariant is not a slice-torus link invariant by itself, but we need to re-scale it and add a correction term; more specifically, recall that the slice-torus link invariant associated to the Rasmussen invariant is
\[\nu_{s}(L) = \dfrac{s(L)+\ell-1}{2}\in \mathbb{Z}\:,\] 
where $L$ is an $\ell$-component link.  Which means that the values that $s(W_+(H_{\pm},t_{1}, t_{2}))$ may assume are $3$, $1$ and $-1$. The KnotTheory package of \textsf{Mathematica} (\cite{Mathematica}) was used to compute the Khovanov homology, and the results of these computations are collected in the appendix.

Let us start with the link $W_{+}(H_\pm,0,1)$. In this case, the result of our computations is the following (cf. Table \ref{figure:Kh+(0,1)}) \[\dim\:\text{Gr}_{\mathcal{F}}^{j}H_{Lee}^{0}(W_+(H_\pm,0,1))=\left\{\begin{aligned}
&2\:\:\:\:\:\text{if }j=0,2\\
&0\:\:\:\:\:\text{otherwise}\end{aligned}\right.\:.\]
It follows that $s(W_{+}(0,1)) = 1$, and thus
\[F_{\nu_{s}}(H_{\pm})(0,1) = \frac{1 + 2 - 1}{2} = 1.\]
Notice that the link $W_+(H_\pm,0,1)$ is \emph{pseudo-thin}, that is $\text{Gr}_{\mathcal{F}}^{\bullet}H_{Lee}^{0}$ is supported in two points (see \cite{Cavallo4}). 

Now, let us consider the links $W_{+}(H_\pm ,1,1)$ and $W_{+}(H_\pm ,0,2)$. We have that (cf. Table \ref{figure:Kh+(0,2)})
\[\dim\:\text{Gr}_{\mathcal{F}}^{j}H_{Lee}^{0}(W_+(H_\pm,0,2))=
\dim\:\text{Gr}_{\mathcal{F}}^{j}H_{Lee}^{0}(W_+(H_\pm,1,1))=\left\{\begin{aligned}
&2\:\:\:\:\:\text{if }j=0\\
&1\:\:\:\:\:\text{if }j=-2,2\\
&0\:\:\:\:\:\text{otherwise}\end{aligned}\right.\:.\]
It follows that $F_{\nu_{s}}(H_{\pm})(1,1) = F_{\nu_{s}}(H_{\pm})(1,2) = 0$.
Notice that neither of these links is pseudo-thin (cf. Remark \ref{remark:noPardon}).

Thanks to the non-increasing property proved in Theorem \ref{teo:decreasing}, $F_{\nu_s}$ is completely determined. The result of our computations is shown in the left hand side of Figure \ref{figure:F_s}.
\begin{figure}[h]
    \centering

\begin{tikzpicture}[scale=.6]
\draw[opacity=.3, gray] (-5.5,-5.5) grid (5.5,5.5);
\draw[gray,  ->] (-5.5,0) -- (5.5,0);
\draw[gray, ->] (0,-5.5) -- (0,5.5);

\draw[fill, opacity=.2, orange] (-5.25,5.25) --(-5.25,-1) -- (-1,-1)--(-1,-5.25) --(5.25,-5.25) --(5.25,-1) --(0,-1)-- (0,0) -- (-1,0) --  (-1,5.25) -- cycle;
\draw[yellow, fill, opacity = .25] (5.25,0) rectangle (1,-1);
\draw[fill, opacity=.2, orange] (0,1) rectangle (-1,0);
\draw[yellow, fill, opacity = .25] (0,5.25) rectangle (-1,1);
\draw[fill, opacity=.2, orange] (1,0) rectangle (0,-1);
\draw[fill, opacity=.5, red] (-1,-5.25) rectangle (-5.25,-1);

\draw[fill, opacity=.25,yellow] (0,0) rectangle (5.25,5.25);
\node[gray, above left] at (0,5.25) {$t_{2}$};
\node[gray, below right] at (5.25,0) {$t_{1}$};
\node[] at (0,-6) {$F_{\nu_s}(H_{\pm})(t_{1},t_{2})$};

\foreach \x in {-5,...,-2}
 {
 \node[draw,star,star points=3, yellow,fill,scale=.25] at (-\x,0) {};
 \node[draw,star,star points=3, yellow,fill,scale=.25] at (0,-\x) {};
 }

\foreach \x in {-5,...,-1}
  {

  \foreach \y in {-5,...,-1}
  {
  \node[draw,rectangle, red,fill,scale=.35] at (\x,\y) {};
  \node[draw,star,star points=3, yellow,fill,scale=.25] at (-\x,-\y) {};
  }
    \foreach \y in {0,...,5}
  {
  \node[draw,star,star points=7, orange,fill,scale=.25] at (\x,\y) {};
  \node[draw,star,star points=7, orange,fill,scale=.25] at (\y,\x) {};
  }
  }
  
\node[draw,star,star points=7, orange,fill,scale=.25] at (0,0) {};
\node[draw,star,star points=7, orange,fill,scale=.25] at (1,0) {};
\node[draw,star,star points=7, orange,fill,scale=.25] at (0,1) {};

\node[draw,star,star points=3,yellow,fill,scale=.5, left] at (2.5,6.5) {};
\node[right] at (2.5,6.5) {= 0};
\node[draw,star,star points=7, orange,fill,scale=.5, left] at (-.5,6.5) {};
\node[right] at (-.5,6.5) {= 1};
\node[draw,rectangle, red,fill,scale=.7, left] at (-4.5,6.5) {};
\node[right] at (-4.5,6.5) {= 2};
\end{tikzpicture}
\hspace{.5cm}
\begin{tikzpicture}[scale=.6]
\draw[opacity=.3, gray] (-5.5,-5.5) grid (5.5,5.5);
\draw[gray,  ->] (-5.5,0) -- (5.5,0);
\draw[gray, ->] (0,-5.5) -- (0,5.5);

\draw[teal, fill, opacity = .25] (5.25,0) rectangle (1,1);
\draw[teal, fill, opacity = .25] (0,5.25) rectangle (1,1);
\draw[fill, opacity=.2, cyan] (-5.25,5.25) --(-5.25,0) -- (0,0) --(0,-5.25) --(5.25,-5.25) --(5.25,0)--(1,0) -- (1,1) --(0,1) -- (0,5.25) -- cycle;

\draw[fill, opacity=.5, blue] (0,-5.25) rectangle (-5.25,0);

\draw[fill, opacity=.25,teal] (1,1) rectangle (5.25,5.25);
\node[gray, above left] at (0,5.25) {$t_{2}$};
\node[gray, below right] at (5.25,0) {$t_{1}$};
\node[] at (0,-6) {$\overline{F}_{\nu_s}(H_{\pm})(t_{1},t_{2})$};

\foreach \x in {-5,...,-1}
  {
  \node[draw,diamond,blue,fill,scale=.25]  at (\x,0) {};
  \node[draw,diamond,blue,fill,scale=.25]  at (0,\x) {};
  
  \foreach \y in {-5,...,-1}
  {
  \node[draw,diamond, blue,fill,scale=.25] at (\x,\y) {};
  \ifthenelse{\x < -1}{\draw[fill,teal] (-\x,-\y) circle (.1);}{\ifthenelse{\y < -1}{\draw[fill,teal] (-\x,-\y) circle (.1);}{}}
  }
    \foreach \y in {0,...,5}
  {
  \node[draw,star, cyan,fill,scale=.25] at (-\x,-\y ) {};
  \node[draw,star, cyan,fill,scale=.25] at (-\y,-\x) {};
  }
  }

\node[draw,diamond,blue,fill,scale=.25] at (0,0) {};
\node[draw,star, cyan,fill,scale=.25] at (1,1) {};

\draw[fill,teal] (3.75,6.5) circle (.2);
\node[right] at (3.8,6.5) {= -2};
\node[draw,star, cyan,fill,scale=.5, left] at (-.5,6.5) {};
\node[right] at (-.5,6.5) {= -1};
\node[draw,diamond, blue,fill,scale=.7, left] at (-5,6.5) {};
\node[right] at (-5,6.5) {= 0};
\end{tikzpicture}
\caption{Computations of the functions $F_{\nu_s}(H_{\pm})(t_{1},t_{2})$ and $\overline{F}_{\nu_s}(H_{\pm})(t_{1},t_{2})$, where $\nu_s$ is the slice-torus link invariant associated to the Rasmussen link invariant.}
\label{figure:F_s}
\end{figure}
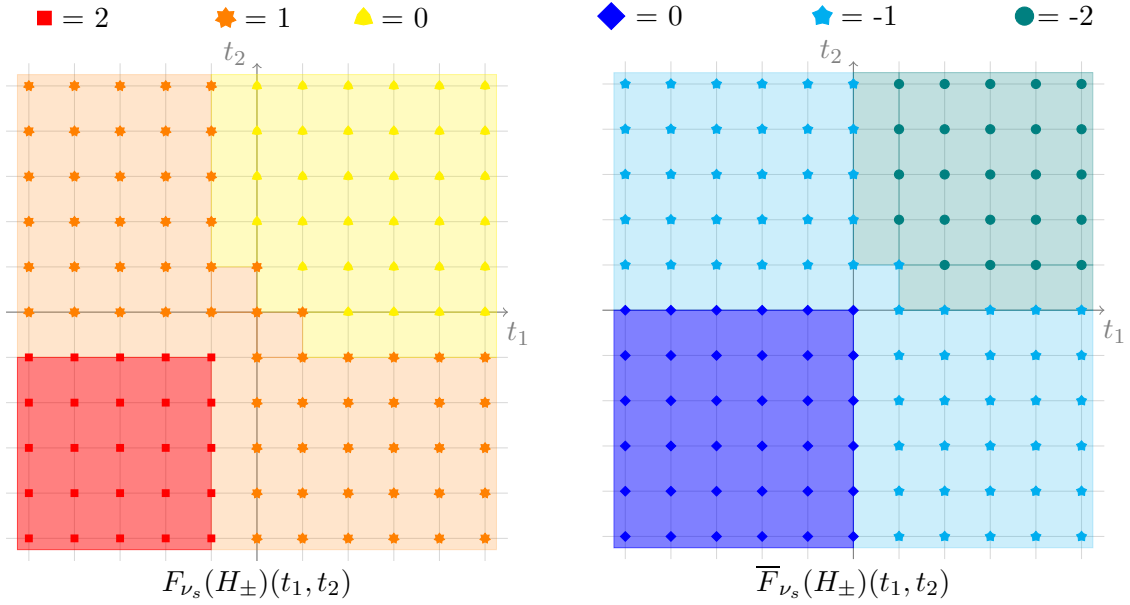

Now, let us turn to the computation of the function $\overline F_{\nu_s}$.
We start by computing the associated graded object to Lee homology for the link $W_-(H_\pm,0,0)$ which yields the following result (cf. Table \ref{figure:Kh-(1,2)}, left)
\[\dim\:\text{Gr}_{\mathcal{F}}^{j}H_{Lee}^{0}(W_-(H_\pm,0,0))=\left\{\begin{aligned}
&2\:\:\:\:\:\text{if }j=-2,0\\
&0\:\:\:\:\:\text{otherwise}\end{aligned}\right.\:.\]
As a consequence, we have that $\overline{F}_{\nu_s}(0,0)=0$, and thus (by Theorem \ref{teo:decreasing}) it follows that $\overline{F}_{\nu_s} (0,t)=\overline{F}_{\nu_s} (t,0)=0$ for each $t\leq 0$.

Finally, we computed the associated graded object to Lee homology for the link $W_-(H_\pm,1,2)$ (cf. Table \ref{figure:Kh-(1,2)}, right), and we obtained that
\[\dim\:\text{Gr}_{\mathcal{F}}^{j}H_{Lee}^{0}(W_-(H_\pm,1,2))=\left\{\begin{aligned}
&2\:\:\:\:\:\text{if }j=-4\\
&1\:\:\:\:\:\text{if }j=-6,-2\\
&0\:\:\:\:\:\text{otherwise}\end{aligned}\right.\:.\]
Now, Proposition \ref{teo:s} impies that $s(W_-(H_\pm,1,2))=-5$, and thus $\overline{F}_{\nu_s}(1,2)  = -2$. This completes the computation of the function $\overline{F}_{\nu_s}$, which is summarized on the right hand side of Figure \ref{figure:F_s}.

\begin{remark}
Notice that while the invariants $t_{\nu}$ and $\overline{t}_{\nu}$ contain the same information, in the case of multi-component links the set of points where $F_{\nu}$ and $\overline{F}_{\nu}$ change their values are not related \emph{a priori}. In Figure \ref{figure:F_s}, we have an example of how the ``jumping loci'' of these the two functions are not trivially related.
\end{remark}

\section{Further examples}\label{section:examplesII}

In this section we explore two further examples: the split links and the link L8a9. By analysing the former example we are able to define a new obstruction for a link to be strongly concordant to a split link. With the latter example we shall see that the functions $F^\prime_{\nu}$ and $\overline{F}^\prime_{\nu}$ contain different information than the linking matrix and $\nu$. 

\subsection{Split links}
Let $L_{i}$ be an $\ell_{i}$-component link, for $i\in \{ 1,2\}$.
Since Whitehead doubling and disjoint union commute, Property (B) tells us that 
\[\nu(W_\pm(L_1\sqcup L_2,(\underline t,\underline s)))=\nu(W_\pm(L_1,\underline t))+\nu(W_\pm(L_2,\underline s))\] for each $\underline t\in\Z^{\ell_1}$ and $\underline s\in\Z^{\ell_2}$. Denote by $L$ the disjoint union of $L_{1}$ and $L_{2}$.
Proposition \ref{prop:step} implies that there exists a unit square in $\Z^{\ell_{1} + \ell_{2}}$ with vertices 
\[ (\underline{t}, \underline{s}),\ (\underline{t} + \underline{e}_{i}, \underline{s} ),\ (\underline{t}, \underline{s} + \underline{e}_{j}),\ \text{and}\ (\underline{t}   + \underline{e}_{i}, \underline{s} + \underline{e}_{j}),\] 
for some $i,\: j\in \{ 1,...,\ell_{1} + \ell_{2} \}$, where $\underline{e}_{i}$ denotes the vector (of the appropriate length) with $i$-th entry $1$ and all the other entries $0$, satisfying the following properties
\begin{itemize}
    \item[$\triangleright$] $F_{\nu}(L)$ assumes at least three different values on the vertices of the square;
    \item[$\triangleright$] the maximal and the minimal values of $F_{\nu}(L)$ on the square are attained exactly once; 
    \item[$\triangleright$] the maximal value of $F_{\nu}(L)$ on the square is equal to $\ell$, where $\ell$ is the number of the components of $L$.
\end{itemize}
We call such a square a \emph{3-valued square} for the function $F_{\nu}(L)$. 
The concept of $3$-valued square can be generalized as follows.
\begin{defin}
Fix $\ell \geq 1$, and consider a bounded function $F:\mathbb{Z}^\ell \to \mathbb{R}$. A \emph{$(k+1)$-valued cube} for $F$ is a $k$-dimensional cube in $\mathbb{Z}^\ell$ with edges of length $1$, such that:
\begin{itemize}
    \item[$\triangleright$] $F$ assumes at least $k+1$ different values on the vertices of the cube.
    \item[$\triangleright$] The maximal and the minimal values of $F$ on the cube are attained exactly once; 
    \item[$\triangleright$] The maximal value of $F$ on the cube is equal to the maximal value of $F$.
\end{itemize}
\end{defin}
Notice that a $3$-valued square is a $3$-valued cube.
The existence of a $3$-valued square for $F_{\nu}(L)$ and $\overline F_{\nu}(L)$ in the case $L$ is the split union of two links can be generalized as follows. 
\begin{prop}
 \label{prop:square}
 Suppose that $L$ is a link with $\ell_s $ split components. Then, $F_{\nu}(L)$ and $\overline F_{\nu}(L)$ have at least one $(\ell_s +1)$-valued cube, for each slice-torus link invariant $\nu$. 
\end{prop}
\begin{proof}
 It follows from Proposition \ref{prop:step}, and from the additivity of $F_{\nu}(L)$ and $\overline F_{\nu}(L)$ with respect to the disjoint union of links.
\end{proof}
This gives a criterion to obstruct the strong concordance with split links. In particular, this criterion requires only a partial computation of either $F_{\nu}(L)$ or $\overline F_{\nu}(L)$ (cf. Remark~\ref{rem:conctosplit}).
\begin{teo}\label{teo:split}
Let $L$ be a link. If there exists a slice-torus link invariant $\nu$ such that either $F_{\nu}(L)$ or $\overline F_{\nu}(L)$ do not admit any $(r+1)$-valued cube, then $L$ is not strongly concordant to any link with $r$ split components. 
\end{teo}
\begin{proof}
 It is an immediate consequence of Theorem \ref{teo:concordant}, which states that $F_{\nu}$ and $\overline F_{\nu}$ are strong concordance invariants, and Proposition \ref{prop:square}.
\end{proof}
In particular, when $\nu$ is a $\Z$-valued slice-torus link invariant, we have that if an $\ell$-component link $L$ is strongly concordant to the disjoint union of $\ell$ knots then $F_{\nu}(L)$ has exactly one $(\ell+1)$-valued cube. 
Moreover, in this case $F_{\nu}$ also assumes the minimal value. 

\subsection{The link L8a9}
In this subsection we use the function $F^\prime_{\nu_{s}}$ to prove that the link $L = L8a9$ (see Figure \ref{L8a9}) is not strongly concordant to the positive Hopf link $H_+$.

Notice that $L$ is a non-split, alternating link with the same signature and linking matrix as $H_+$. 
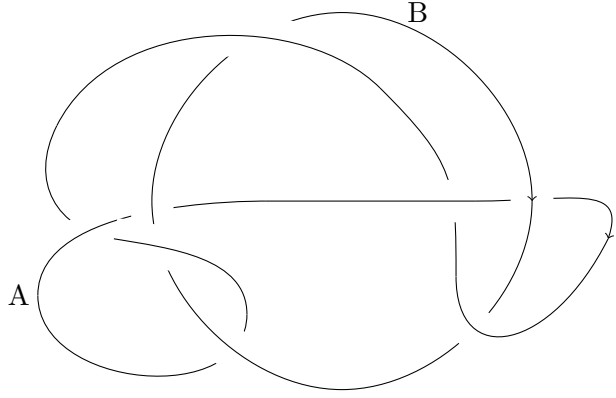
\begin{figure}[H]
  \centering
  \begin{tikzpicture}[scale = .5]

\draw[->] (8.5,0) .. controls +(2.5,0) and +(.5,1.5) .. (12,-1);
\pgfsetlinewidth{40*\pgflinewidth}
\draw[white] (10,0) .. controls +(0,2.5) and +(2.5,0) .. (5,5);\pgfsetlinewidth{.025*\pgflinewidth}

\draw[<-] (10,0) .. controls +(0,2.5) and +(2.5,0) .. (5,5);

\draw (-2.5,2)  .. controls  +(-.75,-1.5) and +(-1.5,.25) .. (-1,-1);

\draw (10,0) .. controls +(0,-2.5) and +(2.5,0) .. (5,-5);

\draw (8,-2) .. controls +(0,2.5) and +(2,-2) .. (6,3);

\draw (-3,-2.5)  .. controls  +(0,-2.5) and +(0,-2.5) .. (2.5,-3);
\pgfsetlinewidth{40*\pgflinewidth}
\draw[white]   (-3,-2.5)  .. controls  +(0,2.5) and +(-.5,-.5) .. (3,0) ;\pgfsetlinewidth{.025*\pgflinewidth}

\draw (-3,-2.5)  .. controls  +(0,2.5) and +(-.5,0) .. (3,0) ;

\pgfsetlinewidth{40*\pgflinewidth}
\draw[white]  (3,0)  .. controls  +(.5,0) and +(-.75,0) .. (8.5,0);\pgfsetlinewidth{.025*\pgflinewidth}
\draw (3,0)  .. controls  +(.5,0) and +(-.75,0) .. (8.5,0);

\pgfsetlinewidth{40*\pgflinewidth}
\draw[white]  (0,0) .. controls +(0,-2.5) and +(-2.5,0) .. (5,-5);\pgfsetlinewidth{.025*\pgflinewidth}
\draw (0,0) .. controls +(0,-2.5) and +(-2.5,0) .. (5,-5);

\pgfsetlinewidth{40*\pgflinewidth}
\draw[white]  (8,-2) .. controls +(0,-2.5) and +(-1.5,-3) .. (12,-1);\pgfsetlinewidth{.025*\pgflinewidth}
\draw (8,-2) .. controls +(0,-2.5) and +(-1.5,-3) .. (12,-1);

\pgfsetlinewidth{40*\pgflinewidth}
\draw[white]  (2.5,-3)  .. controls  +(0,1.5) and +(1.5,-.25) .. (-1,-1);\pgfsetlinewidth{.025*\pgflinewidth}
\draw (2.5,-3)  .. controls  +(0,1.5) and +(1.5,-.25) .. (-1,-1);
\draw (0,0) .. controls +(0,2.5) and +(-2.5,0) .. (5,5);
\pgfsetlinewidth{40*\pgflinewidth}
\draw[white]   (-2.5,2)  .. controls  +(1.5,3) and +(-2,2) .. (6,3);\pgfsetlinewidth{.025*\pgflinewidth}
\draw (-2.5,2)  .. controls  +(1.5,3) and +(-2,2) .. (6,3);
\node at (-3.5,-2.5) {A};
\node at (7,5) {B};
\end{tikzpicture}  
  \caption{A diagram of the link $L$, which is the link $L8a9$ in KnotAtlas.}
  \label{L8a9}
\end{figure}
This means that these two links have also the same Lee homology and the same filtered link Floer homology (see \cite{Cavallo4,Cavallo,Lee05}). It follows that these links have the same $s$ and $\tau$ invariants.
\begin{prop}\label{prop:L8a9andHopf}
 The function $F^\prime_{\nu_{s}}(L;A)$ differs from $F^\prime_{\nu_{s}}(H_\pm)$, and thus there is no strong concordance between the links $L$ and $H_+$. 
\end{prop}
\begin{proof}
 Let us consider the reduced Whitehead doubles $W^\prime_+(L,1;A)$ and $W^\prime_+(H_\pm,1)$. We omitted the component in the latter reduced Whitehead double since the choice of the component for the Hopf link is immaterial (as the results are isotopic).

 Similarly to what we did in the previous subsection, the Khovanov homologies shown in Table \ref{figure:Kh-reduced-L8a9} allow us to determine that 
 \[\dim\:\text{Gr}_{\mathcal{F}}^{j}H_{Lee}^{0}(W_+(L,1;A))=\left\{\begin{aligned}
 &2\:\:\:\:\:\text{if }j=0,2\\
 &0\:\:\:\:\:\text{otherwise}\end{aligned}\right.\]
 and
 \[\dim\:\text{Gr}_{\mathcal{F}}^{j}H_{Lee}^{0}(W_+(H_\pm,1))=\left\{\begin{aligned}
 &2\:\:\:\:\:\text{if }j=0\\
 &1\:\:\:\:\:\text{if }j=-2,2\\
 &0\:\:\:\:\:\text{otherwise}\end{aligned}\right.\]
 It follows from Theorem \ref{teo:s} that $F_{\nu_{s}}^\prime(L;A)(1)=(s(W_+^\prime(L,1;A)) + 1)/2=1$ and $F_{\nu_{s}}^\prime(H_\pm)(1)=(s(W_+^\prime(H_\pm,1)) + 1)/2 = 0$.
\end{proof}
Proposition \ref{prop:L8a9andHopf} tells us that the functions $F_\nu,\overline F_\nu,F^\prime_\nu$ and $\overline F^\prime_\nu$ can effectively give more information, as concordance invariants, than $\nu$ and the linking matrix. In particular, we have also shown that $F^\prime_{\nu_s}(L)$ is not determined by the Lee homology of $L$.

\begin{proof}[Proof of Theorem \ref{teo:Fprimeandnu}]
The result follows from Proposition \ref{prop:L8a9andHopf}. 
\end{proof}

\appendix

\section{Tables of Khovanov homology}
\subsection{Tables relative to Section \ref{section:example}}
In this subsection we collected the tables of the Khovanov homology used in the computations in Section \ref{section:example}. We have highlighted in red the column corresponding to the homological degree $0$, which is the homological degree where Lee homology is concentrated in these cases.
All the homologies are computed with coefficients in $\mathbb{Q}$.
\begin{table}[H]
\centering
\begin{tikzpicture}[scale = .4]
\draw[red, fill, opacity = .3] (0,-5) rectangle (1,6);
\draw[white] (0,-5) rectangle (1,6);
\node at (-5.5,-4.5){1};
\node at (4.5,5.5){1};
\node at (-4.5,-3.5){1};
\node at (3.5,4.5){1};
\node at (-4.5,-2.5){1};
\node at (-3.5,-2.5){1};
\node at (3.5,3.5){1};
\node at (2.5,3.5){2};
\node at (1.5,3.5){1};
\node at (-3.5,-1.5){1};
\node at (-2.5,-1.5){2};
\node at (2.5,2.5){1};
\node at (-1.5,-1.5){2};
\node at (1.5,2.5){1};
\node at (0.5,2.5){1};
\node at (-2.5,-0.5){1};
\node at (-1.5,-0.5){1};
\node at (1.5,1.5){2};
\node at (-0.5,1.5){1};
\node at (0.5,1.5){3};
\node at (-1.5,0.5){2};
\node at (-0.5,0.5){3};
\node at (0.5,0.5){3};
\draw[opacity=.3, gray] (-6,-5) grid (5,6);
\draw[->] (-6,-5) -- (6,-5);
\draw[->] (-6,-5) -- (-6,7);
\node[left] at (-6,7) {$qdeg$};
\node[below] at (6,-5) {$hdeg$};
\foreach \j in {-5,...,5}{\node at (-7, \j  + 0.5) {\pgfmathparse{int(\j*2 + 0)} \pgfmathresult};}
\foreach \j in {-6,...,4}{\node at ( \j + .5 ,-6) {\j};}
\end{tikzpicture}
\caption{The Khovanov homology of the link $W_+(H_\pm,0,1)$.}
\label{figure:Kh+(0,1)}
\end{table}
\begin{table}[H]
    \centering
\begin{tikzpicture}[scale = .4]
\draw[fill, red, opacity=.3] (0,-7) rectangle (1,6);
\draw[white] (0,-7) rectangle (1,6);
\node at (-7.5,-6.5){1};
\node at (-6.5,-5.5){1};
\node at (-6.5,-4.5){1};
\node at (-5.5,-4.5){1};
\node at (4.5,5.5){1};
\node at (-5.5,-3.5){1};
\node at (-4.5,-3.5){2};
\node at (-3.5,-3.5){1};
\node at (-4.5,-2.5){1};
\node at (-3.5,-2.5){1};
\node at (-2.5,-2.5){1};
\node at (-2.5,-2.5){1};
\node at (2.5,3.5){2};
\node at (-1.5,-2.5){1};
\node at (-3.5,-1.5){2};
\node at (-2.5,-1.5){2};
\node at (2.5,2.5){1};
\node at (-1.5,-1.5){2};
\node at (1.5,2.5){2};
\node at (-2.5,-0.5){1};
\node at (-1.5,-0.5){2};
\node at (1.5,1.5){2};
\node at (-0.5,1.5){1};
\node at (-0.5,-0.5){2};
\node at (0.5,1.5){2};
\node at (0.5,-0.5){1};
\node at (-1.5,0.5){1};
\node at (1.5,0.5){1};
\node at (-0.5,0.5){2};
\node at (0.5,0.5){4};
\node at (3.5,3.5){1};
\draw[opacity=.3, gray] (-8,-7) grid (5,6);
\draw[->] (-8,-7) -- (-8,7);
\draw[->] (-8,-7) -- (6,-7);
\node[left] at (-8,7) {$qdeg$};
\node[below] at (6,-7) {$hdeg$};
\foreach \j in {-7,...,5}{\node at (-9, \j  + 0.5) {\pgfmathparse{int(\j*2 + 0)} \pgfmathresult};}
\foreach \j in {-8,...,4}{\node at ( \j + .5 ,-8) {\j};}
\end{tikzpicture}    
\hspace{0.5cm}    
\begin{tikzpicture}[scale = .4]
\draw[fill, red, opacity=.3] (0,-7) rectangle (1,4);
\draw[white] (0,-7) rectangle (1,4);

\node at (-7.5,-6.5){1};
\node at (-6.5,-5.5){1};
\node at (-6.5,-4.5){1};
\node at (-5.5,-4.5){1};
\node at (-5.5,-3.5){1};
\node at (-4.5,-3.5){2};
\node at (-3.5,-3.5){2};
\node at (-4.5,-2.5){1};
\node at (-3.5,-2.5){1};
\node at (2.5,3.5){2};
\node at (-3.5,-1.5){2};
\node at (-2.5,-1.5){3};
\node at (-1.5,-1.5){2};
\node at (-2.5,-0.5){1};
\node at (-1.5,-0.5){1};
\node at (-0.5,-0.5){2};
\node at (-0.5,1.5){1};
\node at (-0.5,-0.5){2};
\node at (0.5,1.5){3};
\node at (0.5,-0.5){1};
\node at (-1.5,0.5){1};
\node at (-0.5,0.5){2};
\node at (0.5,0.5){2};
\node at (1.5,1.5){2};
\draw[opacity=.3, gray] (-8,-7) grid (3,4);
\draw[->] (-8,-7) -- (4,-7);
\draw[->] (-8,-7) -- (-8,5);
\node[below] at (4,-7) {$hdeg$};
\node[left] at (-8,5) {$qdeg$};
\foreach \j in {-7,...,3}{\node at (-9, \j  + 0.5) {\pgfmathparse{int(\j*2 + 0)} \pgfmathresult};}
\foreach \j in {-8,...,2}{\node at ( \j + .5 ,-8) {\j};}
\end{tikzpicture}
\caption{The Khovanov homology of the links $W_+(H_\pm,1,1)$ and $W_+(H_\pm,0,2)$.}
\label{figure:Kh+(0,2)}
\end{table}

\begin{table}[H]
    \centering
\begin{tikzpicture}[scale = .32]
\draw[fill, red, opacity=.3] (0,-7) rectangle (1,4);
\draw[white] (0,-7) rectangle (1,4);
\node at (-5.5,-6.5){1};
\node at (-4.5,-5.5){1};
\node at (-4.5,-4.5){1};
\node at (-3.5,-4.5){1};
\node at (-2.5,-4.5){1};
\node at (-3.5,-3.5){1};
\node at (-2.5,-3.5){2};
\node at (-1.5,-3.5){1};
\node at (4.5,3.5){1};
\node at (-2.5,-2.5){1};
\node at (-1.5,-2.5){2};
\node at (-0.5,-2.5){1};
\node at (3.5,2.5){1};
\node at (-1.5,-1.5){2};
\node at (-0.5,-1.5){2};
\node at (0.5,-1.5){2};
\node at (3.5,1.5){1};
\node at (2.5,1.5){1};
\node at (1.5,-0.5){1};
\node at (-0.5,-0.5){1};
\node at (0.5,-0.5){3};
\node at (2.5,0.5){1};
\node at (1.5,0.5){2};
\node at (0.5,0.5){3};
\draw[opacity=.3, gray] (-6,-7) grid (5,4);
\draw[->] (-6,-7) -- (-6,5);
\draw[->] (-6,-7) -- (6,-7);
\node[left] at (-6,5) {$qdeg$};
\node[below] at (6,-7) {$hdeg$};
\foreach \j in {-7,...,3}{\node at (-7, \j  + 0.5) {\pgfmathparse{int(\j*2 + 0)} \pgfmathresult};}
\foreach \j in {-6,...,4}{\node at ( \j + .5 ,-8) {\j};}
\end{tikzpicture}
\hspace{0.5cm}
\begin{tikzpicture}[scale = .46]
\draw[fill, red, opacity=.3] (0,-13) rectangle (1,0);
\draw[white] (0,-13) rectangle (1,0);
\node at (-11.5,-12.5){1};
\node at (-10.5,-11.5){1};
\node at (-10.5,-10.5){1};
\node at (-9.5,-10.5){1};
\node at (-9.5,-9.5){1};
\node at (-8.5,-9.5){2};
\node at (-7.5,-9.5){1};
\node at (-8.5,-8.5){1};
\node at (-7.5,-8.5){1};
\node at (-6.5,-8.5){1};
\node at (-7.5,-7.5){2};
\node at (-6.5,-7.5){2};
\node at (-6.5,-6.5){1};
\node at (-5.5,-6.5){2};
\node at (-4.5,-6.5){3};
\node at (-5.5,-5.5){1};
\node at (-3.5,-5.5){1};
\node at (-2.5,-5.5){1};
\node at (-4.5,-4.5){1};
\node at (-3.5,-4.5){3};
\node at (-2.5,-4.5){2};
\node at (-2.5,-3.5){1};
\node at (-1.5,-3.5){2};
\node at (-1.5,-2.5){2};
\node at (0.5,-2.5){1};
\node at (-0.5,-1.5){1};
\node at (0.5,-1.5){2};
\node at (0.5,-0.5){1};
\draw[opacity=.3, gray] (-12,-13) grid (1,0);
\draw[->] (-12,-13) -- (-12,1);
\draw[->] (-12,-13) -- (2,-13);
\node[left] at (-12,1) {$qdeg$};
\node[below] at (2,-13) {$hdeg$};
\foreach \j in {-13,...,-1}{\node at (-13, \j  + 0.5) {\pgfmathparse{int(\j*2 + 0)} \pgfmathresult};}
\foreach \j in {-12,...,0}{\node at ( \j + .5 ,-14) {\j};}
\end{tikzpicture}
\caption{The Khovanov homology of the links $ W_-(H_\pm, 0,0) $ and $ W_-(H_\pm,1,2 )$.}
\label{figure:Kh-(1,2)}
\end{table}

\subsection{Tables relative to Section \ref{section:examplesII}}
In this subsection we collected the tables of the Khovanov homology used in the computations in Section \ref{section:examplesII}.
\begin{table}[H]
    \centering
    \begin{tikzpicture}[scale = .425]
\draw[fill, red, opacity=.3] (0,-5) rectangle (1,5);
\draw[white] (0,-5) rectangle (1,5);
\node at (0.5,0.5){2};
\node at (0.5,1.5){2};
\node at (0.5,2.5){1};
\node at (-5.5,-4.5){1};
\node at (-4.5,-2.5){1};
\node at (-3.5,-2.5){1};
\node at (-2.5,-1.5){1};
\node at (-2.5,-0.5){1};
\node at (-1.5,0.5){1};
\node at (-1.5,-1.5){1};
\node at (-0.5,0.5){2};
\node at (1.5,1.5){1};
\node at (2.5,2.5){1};
\node at (2.5,3.5){1};
\node at (3.5,4.5){1};
\draw[opacity=.3, gray] (-6,-5.0) grid (4,5.0);
\draw[->] (-6,-5) -- (-6,6);
\draw[->] (-6,-5) -- (5,-5);
\node[left] at (-6,6) {$qdeg$};
\node[below] at (5,-5) {$hdeg$};
\foreach \j in {-5.0,...,4.0}{\node at (-7, \j  + 0.5) {\pgfmathparse{int(\j*2 + 0)} \pgfmathresult};}
\foreach \j in {-6,...,3}{\node at ( \j + .5 ,-6.0) {\j};}
\end{tikzpicture}
\hspace{0.5cm}
\begin{tikzpicture}[scale = .5]
\draw[fill, red, opacity=.3] (0,-4) rectangle (1,4);
\draw[white] (0,-4) rectangle (1,4);
\node at (0.5,0.5){2};
\node at (0.5,-0.5){1};
\node at (0.5,1.5){2};
\node at (-3.5,-3.5){1};
\node at (-2.5,-1.5){1};
\node at (-1.5,-1.5){1};
\node at (-0.5,0.5){1};
\node at (-0.5,-0.5){1};
\node at (1.5,1.5){1};
\node at (2.5,3.5){1};
\draw[opacity=.3, gray] (-4,-4.0) grid (3,4.0);
\draw[->] (-4,-4) -- (-4,5);
\draw[->] (-4,-4) -- (4,-4);
\node[left] at (-4,5) {$qdeg$};
\node[below] at (5,-4) {$hdeg$};
\foreach \j in {-4.0,...,3.0}{\node at (-5, \j  + 0.5) {\pgfmathparse{int(\j*2 + 0)} \pgfmathresult};}
\foreach \j in {-4,...,2}{\node at ( \j + .5 ,-5.0) {\j};}
\end{tikzpicture}
 \caption{The Khovanov homology of the links $ W^\prime_+(L,1;A)$ and $ W^\prime_+(H_\pm,1)$.}
\label{figure:Kh-reduced-L8a9}
\end{table}

\end{document}